\documentclass[12pt,a4paper]{amsart}

 \usepackage{amssymb, amstext, amscd, amsmath, amsfonts, amsthm, amscd, color}

\usepackage{tikz,mathdots,enumerate}

\usepackage{graphicx, array, blindtext}
\usepackage{url}
\usepackage{tikz}
\usetikzlibrary{shapes.geometric, calc}
\usepackage{caption}
\usepackage{array}
\usepackage{epsfig}
\usepackage{eucal}
\usepackage{latexsym}
\usepackage{mathrsfs}
\usepackage{textcomp}
\usepackage{verbatim}

\newtheorem{thm}{Theorem}[section]

\newtheorem{cor}[thm]{Corollary}

\newtheorem{lemma}[thm]{Lemma}

\newtheorem{prop}[thm]{Proposition}

\numberwithin{equation}{section}

\theoremstyle{definition}

\newtheorem{rem}[thm]{Remark}

\newtheorem{definition}[thm]{Definition}
\newcommand{\bR}{{\mathbb{R}}}

\newcommand{\bZ}{{\mathbb{Z}}}

  \newcommand{\A}{{\mathcal{A}}}

  \newcommand{\F}{{\mathcal{F}}}

  \newcommand{\U}{{\mathcal{U}}}

\usepackage{verbatim}

\usepackage{tikz}
\usetikzlibrary{calc}

\begin{document}


\title[The norm closed triple semigroup algebra]{The norm closed triple semigroup algebra}



\author[E. Kastis]{E. Kastis}
\address{Dept.\ Math.\ Stats.\\ Lancaster University\\
Lancaster LA1 4YF \\U.K. }

\email{l.kastis@lancaster.ac.uk}



\begin{abstract}
The $w^\ast$-closed triple semigroup algebra was introduced by Power and the author in \cite{kas-pow}, where it was proved to be reflexive and to be chiral, in the sense of not being unitarily equivalent to its adjoint algebra. Here an analogous operator norm-closed triple semigroup algebra $A_{ph}^{G_+}$ is considered and shown to be a triple semi-crossed product
for the action on analytic almost periodic functions by the semigroups of one-sided translations and one-sided dilations. The structure of isometric automorphisms of $A_{ph}^{G_+}$ is determined and $A_{ph}^{G_+}$ is shown to be chiral with respect to isometric isomorphisms.
\end{abstract}
\date{}

\maketitle

\section{Introduction}
Let $\{M_\lambda\, :\, \lambda\in\bR\}$, $\{D_\mu\,:\, \mu\in \bR\}$ and $\{V_t\,:\, t\in\bR\} $ be the unitary operators of multiplication, translation and dilation respectively, acting on the Hilbert space $L^2(\bR)$ given by
\begin{align*}
M_\lambda f(x)= e^{i\lambda x} f(x),~~D_\mu f(x)=f(x-\mu),~~V_tf(x)=e^{t/2}f(e^tx).
\end{align*}
The first two groups give a celebrated irreducible representation of the Weyl commutation relations  in the form
$M_\lambda D_\mu =e^{i\lambda\mu} D_\mu M_\lambda$, while the dilation group satisfies the relations
\begin{align*}
V_t M_\lambda=M_{e^t\lambda}V_t ~\text{ and }~ 
V_t D_\mu=D_{e^{-t}\mu}V_t.
\end{align*}
Our main results are the determination of the isometric isomorphism group of the norm closed nonselfadjoint operator algebra $A_{ph}$, generated by the semigroups for $\lambda,\mu, t\geq0$, and a chirality property for $A_{ph}$. 
The weakly closed operator algebra $\A_{ph}$ generated by the three semigroups was shown in \cite{kas-pow} to be reflexive, in the sense of Halmos \cite{rad-ros}, and to have the rigidity property of failing to be unitarily equivalent to the adjoint operator algebra $\A_{ph}^\ast$. We termed this property a chiral property since, in many other respects the algebras carry similar properties. In particular, the invariant projection lattices $Lat \A_{ph}$ and $Lat\A_{ph}^\ast$ were naturally order isomorphic lattices with unitarily equivalent pairs of interval projections. That $\A_{ph}$ is chiral contrasts with the parabolic algebra $\A_p$ \cite{kat-pow-1}, generated by the multiplication and translation semigroups, as well as with the usual Volterra nest algebra on $L^2(\bR)$ \cite{dav1}. For related classes of semigroup generated weakly closed algebras see also \cite{ano-kat-tod, kat-pow-2, lev-pow-1}.

In the norm closed case considered here we take advantage of the theory of discrete semicrossed products. We prove that there natural identifications
\begin{align*}
A_p=AAP\rtimes_\tau\bR_+~,~ A_{ph}^{\bZ_+}=A_p\rtimes_v \bZ_+~,~ A_{ph}^{\bR_+}=A_p\rtimes_v \bR_+
\end{align*}
where $A_p$ is the norm closed parabolic algebra, $AAP$ is the algebra of analytic almost periodic functions in $L^\infty(\bR)$ and $A_{ph}^{G_+}$ is generated by $A_p$ and $\{V_t:t\in G_+\}$.
The notion of semicrossed products began with Arveson \cite{arv} in 1967, and was developed by the studies of Peters \cite{pet} and McAsey and Muhly \cite{mca-muh} in the early eighties.
Since then, several studies of semicrossed products of $C^\ast$-algebras have been under investigation by various authors \cite{pow1, dav-kats1, kak-kats}.
To avoid categorical issues we define all the semicrossed products algebras that we consider as subalgebras of their associated $C^\ast$-crossed products \cite{ped}.
Indeed, in the case of the semicrossed product $A_p$, this algebra coincides with its universal counterpart, defined as usual in terms of all contractive covariant representations of the generator semigroup \cite{pow2}. However, we do not know if this persists for the triple semicrossed product algebra $A_{ph}^{G_+}$.

 Many of the results of isomorphisms of crossed products are concerned with the case of the discrete group $\bZ$ (see \cite{pow1}), while we also deal with the group of the real numbers endowed again with the discrete topology. This case is more subtle since the group $C^\ast$-algebra of $\bR$ is the algebra of the almost periodic functions, which induces limit characters that arise from the Bohr compactification of $\bR$. Moreover, the introduction of the triple semigroup semicrossed product makes the identification of the maximal ideal space of the algebra problematic. To overcome such problems, we identify the isometric automorphism group of the norm closed parabolic algebra $A_p$ and prove that each isometric automorphism of $A_{ph}^{G_+}$ leaves $A_p$ invariant. Finally, we remark that this approach works for both cases where we choose the dilation operators $V_t$  in $\bR$ or in $\bZ$.

\section {Preliminaries}

\subsection{Discrete Crossed Products}
Crossed products of C$^\ast$-algebras were introduced by Murray and von Neumann as a tool for studying groups that act on C$^\ast$-algebras as automorhisms, since they provide a larger algebra that encodes both the original C$^\ast$-algebra and the group action. The reader may look for more details in \cite{dav2, pow, broza, wil}.
\begin{definition}
 A \textbf{C}$^\ast$\textbf{-dynamical system} is a triple $(\mathcal{A},G,\alpha)$ that consists of a unital C$^\ast$-algebra $\mathcal{A}$, a group $G$ and a homomorphism 
\begin{align*}
\alpha: G\rightarrow Aut(\mathcal{A}): s\mapsto \alpha_s
\end{align*}   
Given a C$^\ast$-dynamical system, a \textbf{covariant representation} is a pair $(\pi,U)$, such that $\pi$ is a representation of $\mathcal{A}$ on some Hilbert space $H$ and $U:s\mapsto U_s$ is a \textbf{unitary representation} of $G$ on the same space, that also satisfies the formula 
\begin{align*}
U_s\pi(A)U_s^\ast=\pi(\alpha_s(A)),\,\forall A\in\mathcal{A},s\in G.
\end{align*}
\end{definition}
In this section, we will restrict our attention to discrete crossed products, where $G$ is a discrete \textbf{abelian} group. 

We form the complex vector space $\mathcal{A}G$ of (generalized) trigonometric polynomials :
\begin{align*}
\mathcal{A}G=\textrm{span}\{\delta_s\otimes A : s\in G,\,A\in \mathcal{A}\},\textrm{ where } 
\delta_s(t)=\begin{cases} 1, & \mbox{if } t=s \\ 0, & \mbox{if } t\neq s \end{cases}
\end{align*}
 and endow it with ring multiplication and involution given by
 \begin{align*}
 (\delta_s\otimes A)\cdot (\delta_t\otimes B)&=
 (\delta_{s+t}\otimes A\alpha_s(B))\\
(\delta_s\otimes A)^\ast&=(\delta_{-s}\otimes \alpha_{-s}(A^\ast))
\end{align*}
respectively.
 The algebra $\mathcal{A}G$ becomes a normed $\ast$-algebra with the norm: 
 \begin{align*}
 \bigg\|\sum_{\substack{s\in F \\ F \subset\subset G}}(\delta_s\otimes A_s)\bigg\|_{\ell_1}=
\sum_{\substack{s\in F \\ F \subset\subset G}}\|A_s\|,
\end{align*} 
where the notation $F\subset\subset G$ means that $F$ is a finite subset of $G$. 

Each $(\pi,U)$ covariant representation induces a $\ast$-homomorphism on $\mathcal{A}G$, since the linear map $\pi\rtimes U$ in $\mathcal{A}G$ with
\begin{align*}
(\pi\rtimes U)\left(\sum_{\substack{s\in F \\ F \subset\subset G}} (\delta_s\otimes A_s)\right)= \sum_{\substack{s\in F \\ F \subset\subset G}}\pi(A_s)U_s
\end{align*}
is bounded:
\begin{align*}
\Bigg\|(\pi\rtimes U)\left(\sum_{\substack{s\in F \\ F \subset\subset G}} (\delta_s\otimes A_s)\right)\Bigg\|\leq 
\sum_{\substack{s\in F \\ F \subset\subset G}}\|\pi(A_s)\|\leq
\sum_{\substack{s\in F \\ F \subset\subset G}}\|(A_s)\|=
\Bigg\|\sum_{\substack{s\in F \\ F \subset\subset G}}(\delta_s\otimes A_s)\Bigg\|_{\ell_1}.
\end{align*}
We define the $C^\ast$-algebra $\mathcal{A}\times_\alpha G$ as the completion of $\mathcal{A}G$ with respect to the norm
\begin{align*}
\|F\|:=\sup\{\|(\pi\rtimes U)(F)\|\,:\,(\pi,U)\textrm{ covariant representation of }\mathcal{A}G\}.
\end{align*}
  
Observe that $\mathcal{A}\times_\alpha G$ satisfies the universal property :

\begin{center}
\textit{ If $(\pi,U)$ is a covariant representation of the dynamical system $(\mathcal{A},G,\alpha)$,\\ then there is a representation $\tilde{\pi}$ of $\mathcal{A}\times_\alpha G$, such that $\tilde{\pi}(\delta_s\otimes A)=\pi(A)U_s$.}
\end{center}

To prove that the crossed product norm  is a C$^\ast$-norm and not just a seminorm, we need a covariant representation that admits a faithful representation of $\mathcal{A}G$. By the Gelfand Naimark theorem, let $\pi$ be a faithful representation of $\mathcal{A}$ on some Hilbert space $H$. Define the covariant representation $(\tilde{\pi},\Lambda)$ of $(\mathcal{A},G,\alpha)$, such that
\begin{equation}\label{lrrep1}
\tilde{\pi}:\mathcal{A}\rightarrow B(\ell^2(G,H)): (\tilde{\pi}(A)x)(s)=
\pi(\alpha_{-s}(A))(x(s))
\end{equation}
and $\Lambda$ is the \textbf{left regular representation} on $\ell^2(G,H)$
\begin{equation}  \label{lrrep2}
\Lambda:G\rightarrow B(\ell^2(G,H)): (\Lambda_t x)(s)=x(s-t)
\end{equation}
for all $s,t\in G,\,A\in\mathcal{A},\,x\in\ell^2(G,H)$. One can easily verify that $\tilde{\pi}$ is a representation of $\mathcal{A}$ and $\Lambda$ is a unitary representation of $G$. Also we have the covariance condition
\begin{align*}
(\Lambda_t\tilde{\pi}(A)\Lambda_t^\ast x)(s)&=(\tilde{\pi}(A)\Lambda_t^\ast x)(s-t)=
\pi(\alpha_{t-s}(A))(\Lambda_t^\ast x(s-t))=\\&=
\pi(\alpha_{-s}\alpha_t(A))(x(s))=(\tilde{\pi}(\alpha_t(A))x)(s).
\end{align*}
Note first that the algebra $\mathcal{A}$ is isometrically embedded into the crossed product by the inclusion map
\begin{align*}
\iota:\mathcal{A}\rightarrow \mathcal{A}\times_\alpha G : A\mapsto (\delta_0\otimes A).
\end{align*}
To prove that $\iota$ is an isometry, let $\delta_{s,\xi}$ be the vector in $\ell^2(G,\mathcal{A})$ that is defined by
\begin{align*}
 \delta_{s,\xi}(t):=\delta_{s,t}\xi=
\begin{cases} \xi, & \mbox{if } t=s \\ 0, & \mbox{if } t\neq s \end{cases}.
\end{align*}
Then
\begin{align*}
\Lambda_t\delta_{s,\xi}=\delta_{s+t,\xi} ~\text{ and }~ \tilde{\pi}(A)\delta_{s,\xi}=\delta_{s,\pi(\alpha_{-s}(A))\xi},
\end{align*}
for every $t,s\in G, A\in\mathcal{A}, \xi\in H$. Calculate
\begin{align*}
\|(\delta_0\otimes A)\|^2&\geq \sup_{\|\xi\|=1} \| (\tilde{\pi}\rtimes \Lambda)(\delta_0\otimes A) \delta_{0,\xi}\|^2_{\ell^2(G,H)} =
\sup_{\|\xi\|=1} \|\tilde{\pi}(A) \delta_{0,\xi}\|^2_{\ell^2(G,H)}=\\&=\sup_{\|\xi\|=1} \|\pi(A)\xi\|^2_H=\|A\|^2.
\end{align*}
The opposite inclusion is straightforward from the fact that $\|A\|=\|(\delta_0\otimes A)\|_{\ell^1}$.

For every $s\in G$, we denote by $V_s$ the operator
\begin{align*} 
V_s:H\rightarrow\ell^2(G,H):\xi\mapsto\delta_{s,\xi}
\end{align*}
so its adjoint operator has the form $V_s^\ast:\ell^2(G,H)
\rightarrow H:x\mapsto x(s)$. 
Given now any element
$\sum\limits_{\substack{s\in F \\ F \subset\subset G}}(\delta_s\otimes A_s)\in\mathcal{A}G$ and $\xi\in H$ we have
\begin{align*}
V_0^\ast(\tilde{\pi}\rtimes\Lambda)\left(\sum\limits_{\substack{s\in F \\ F \subset\subset G}}(\delta_s\otimes A_s)\right)V_0\xi&=
\sum\limits_{\substack{s\in F \\ F \subset\subset G}}
V_0^\ast\tilde{\pi}(A_s)\Lambda_s\delta_{0,\xi}=
\sum\limits_{\substack{s\in F \\ F \subset\subset G}}
V_0^\ast\tilde{\pi}(A_s)\delta_{s,\xi}=\\&=
\sum\limits_{\substack{s\in F \\ F \subset\subset G}}
V_0^\ast\delta_{s,\pi(\alpha_{-s}(A_s))\xi}=
\sum\limits_{\substack{s\in F \\ F \subset\subset G}}
\delta_{s,0}\pi(\alpha_{-s}(A_s))\xi
=\\&=\pi(A_0)\xi.
\end{align*}
Hence it follows readily from the equality $\|\pi(A_0)\|=\bigg\|V_0^\ast(\tilde{\pi}\rtimes\Lambda)\left(\sum\limits_{\substack{s\in F \\ F \subset\subset G}}(\delta_s\otimes A_s)\right)V_0\bigg\|$, that 
$\|A_0\|\leq \|\sum\limits_{\substack{s\in F \\ F \subset\subset G}}(\delta_s\otimes A_s)\|$. Therefore, one can define the contractive map 
\begin{equation}\label{expectinag}
E_0:\mathcal{A}G \rightarrow \mathcal{A}:  \sum
\limits_{\substack{s\in F \\ F \subset\subset G}}(\delta_s\otimes A_s) \mapsto A_0.
\end{equation}
Check also that for every $X=\sum\limits_{\substack{s\in F \\ F \subset\subset G}}(\delta_s\otimes A_s)\in \mathcal{A}G$ we get $E_0(XX^\ast )= \sum\limits_{\substack{s\in F \\ F \subset\subset G}} A_s^\ast A_s$, so the map $E_0$ keeps the cone of positive elements of $\mathcal{A}G$ invariant. So we have proved the following
\begin{prop}\label{conexp}
The map $E_0$ is an expectation \footnote{An expectation of a C$^\ast$-algebra onto a subalgebra is a positive, unital idempotent map.} on $\mathcal{A}G$ and extends by continuity to a map on $\mathcal{A}\times_\alpha G$ with the same properties.
\end{prop}
Define the $t$-th \textbf{Fourier coefficient} of $X \in\mathcal{A}\times_\alpha G$ by
\begin{align*}
E_t(X)=E_0(X(\delta_{-t}\otimes 1))\in\mathcal{A}.
\end{align*}
Note that for every element $X=\sum\limits_{\substack{s\in F \\ F \subset\subset G}}(\delta_s\otimes A_s)\in\mathcal{A}G$ and $t\in G$, we get $E_t(X)=A_t$, and so it follows 
\begin{align*}
X=\sum_{\substack{s\in F \\ F \subset\subset G}}(\delta_s\otimes E_s(X)).
\end{align*}
We can now see that the left regular representation $\tilde{\pi}\rtimes \Lambda$ of $\mathcal{A}G$ is faithful. Given $X=\sum
\limits_{\substack{s\in F \\ F \subset\subset G}}(\delta_s\otimes A_s) \in\mathcal{A}G$ such that $\|(\tilde{\pi}\rtimes \Lambda)(X)\|=0$, then for every $t\in G$ we have
\begin{align*}
\|A_t\|=\|\pi(A_t)\|=\|V_0^\ast(\tilde{\pi}\rtimes \Lambda)(X(\delta_{-t}\otimes 1))V_0\|\leq\|(\tilde{\pi}\rtimes \Lambda)(X)\|=0.
\end{align*}
Therefore $A_t=0$ for every $t\in G$, but this yields that $X=0$.


\begin{rem}\label{reducedcp}
Since the left regular representation is faithful, we can define the reduced crossed product norm on $\mathcal{A}G$
\begin{align*}
\|\cdot\|_r = \|(\tilde{\pi}\rtimes \Lambda) (\cdot)\|.
\end{align*} 
The norm $\|\cdot\|_r$  does not depend on the choice of the faithful representation $\pi$ (see \cite{broza}). The  completion of $\mathcal{A}G$ with respect to the reduced crossed product norm gives rise to the \textbf{reduced crossed product}, denoted by $\mathcal{A}\times_\alpha^r G$. Moreover, repeating the proof of Proposition \ref{conexp}, one can show that the contraction $E_0$ given by the formula \eqref{expectinag} extends to an expectation $\tilde{E}_0$ on $\mathcal{A}\times_\alpha^r G$.
\end{rem}
\begin{rem}
In the general case, the construction via the left regular representation of $G$ is not sufficient to determine the norm of the crossed product. Although in the special case that $G$ is discrete abelian, so amenable\footnote{A group $G$ is called amenable if there is a left translation invariant state on $L^\infty(G)$}, the reduced crossed product equals the full crossed product. In the following subsection, we will give a proof of this claim in the case where $G$ is the discrete group of real numbers. 
\end{rem}

\subsection{Crossed Products by $\mathbb{R}_d$}

From now on, the group $G$ is either $\mathbb{Z}$ or $\mathbb{R}_d$; we use $\mathbb{R}_d$ to denote $\mathbb{R}$ equipped with the discrete topology. The theory about crossed products by $\mathbb{Z}$ can be found in \cite{dav2}. In this section, we develop the theory for $\mathbb{R}_d$.

\begin{prop}\label{countnumfoco}
Let $(\mathcal{A},\mathbb{R}_d,\alpha)$ be a C$^\ast$-dynamical system. Each $X\in \mathcal{A}\times_\alpha\mathbb{R}_d$ has only a countable number of nonzero Fourier coefficients. 
\end{prop}
\begin{proof}
Let $(Y_n)_{n}$ be a sequence of generalized trigonometric polynomials in $\mathcal{A}\times_\alpha\mathbb{R}_d$ such that 
\begin{align*}
\|X-Y_n\|\leq \frac{1}{n}.
\end{align*}
We denote by $\Gamma_n$ the finite set of indices of nonzero Fourier coefficients of $Y_n$ and by $\Gamma$ the set 
\begin{align*}
\cup_{n\in\mathbb{N}} \Gamma_n.  
\end{align*}
The set $\Gamma$ is countable. Suppose now $k \notin \Gamma$; then
\begin{align*}
\|E_k (X)\|\leq \|E_k(X) - E_k(Y_n)\|+\|E_k(Y_n)\|\leq \|X-Y_n\| \leq \frac{1}{n}
\end{align*}
for every $n\in \mathbb{N}$.
\end{proof}

Given a C$^\ast$-dynamical system $(\mathcal{A},\mathbb{R}_d,\alpha)$, 
fix $\lambda\in \mathbb{T}$. Then the map $U:s\mapsto(\delta_s\otimes\lambda^s\cdot 1)$ is a unitary representation of $\mathbb{R}_d$ , such that 
\begin{align*}
U_s\iota(A) U_s^\ast=(\delta_s\otimes\lambda^s\cdot 1) (\delta_0\otimes A)
(\delta_{-s}\otimes\overline{\lambda^s}\cdot 1)=
(\delta_0\otimes\alpha_s(A))=\iota(\alpha_s(A)),
\end{align*}
for every $A\in\mathcal{A}$. Hence the pair $(\iota,U)$ is a covariant representation of $(\mathcal{A},\mathbb{R}_d,\alpha)$, and so the universal property of $\mathcal{A}\times_\alpha \mathbb{R}_d$ induces an automorphism:
\begin{align*}
\phi_\lambda:\mathcal{A}\times_\alpha \mathbb{R}_d\rightarrow \mathcal{A}\times_\alpha \mathbb{R}_d:
(\delta_s\otimes A)\mapsto (\delta_s\otimes \lambda^s A).
\end{align*}
Moreover, given $X\in\mathcal{A}\times_\alpha \mathbb{R}_d$  the map $t\mapsto \phi_{e^{it}}(X)$ is norm continuous for every $t\in \mathbb{R}$; indeed, one can check it first on the unclosed algebra of trigonometric polynomials and extend it to the closure by a standard approximation argument.
 So, given $T>0$,  we can define 
\begin{align*}
\Phi_T(X)=\frac{1}{2T}\int_{-T}^{T} \phi_{e^{it}}(X)dt.
\end{align*}

Check that $\|\Phi_T(X)\|\leq  \frac{1}{2T}\int_{-T}^{T} \|\phi_{e^{it}}(X)\|dt=\|X\|$, so $\|\Phi_T\|\leq 1$.
Given a trigonometric polynomial $Y=\sum\limits_{\substack{s\in F \\ F \subset\subset \mathbb{R}}}(\delta_s\otimes A_s)$ in $\mathcal{A}\times_\alpha \mathbb{R}_d$, we have
\begin{align*}
\Phi_T(Y)&=\frac{1}{2T}\int_{-T}^{T} \phi_{e^{it}}(Y) dt=\\
&=\frac{1}{2T}\int_{-T}^{T} \sum\limits_{\substack{s\in F \\ F \subset\subset \mathbb{R}}}(\delta_s\otimes e^{its}A_s) dt=\\&= \sum\limits_{\substack{s\in F \\ F \subset\subset \mathbb{R}}}
 (\delta_s\otimes A_s) \frac{1}{2T} \int_{-T}^{T} (\delta_0\otimes e^{its}\cdot 1) dt.
\end{align*}
Compute now the limit $\lim\limits_{T\rightarrow \infty} \frac{1}{2T} \int_{-T}^{T} (\delta_0\otimes e^{its}\cdot 1) dt$.
\begin{description} 
	\item [\textbullet $\,\,s=0.$]
	$\lim\limits_{T\rightarrow \infty} \frac{1}{2T} \int_{-T}^{T} (\delta_0\otimes 1) dt=
	\lim\limits_{T\rightarrow \infty} \frac{1}{2T} \cdot 2T (\delta_0\otimes 1) = (\delta_0\otimes 1)$;
	\item [\textbullet $\,\,s\neq 0.$]	
	$\lim\limits_{T\rightarrow \infty} \frac{1}{2T} \int_{-T}^{T} (\delta_0\otimes e^{its}\cdot 1) dt=
		\lim\limits_{T\rightarrow \infty} \frac{1}{2T} \frac{e^{iTs}- e^{-iTs}}{is} (\delta_0\otimes 1) \rightarrow (\delta_0\otimes 0)$, as $T\rightarrow\infty$.
\end{description}
Hence by the linearity of limits, we obtain that 
\begin{align*}
\lim\limits_{T\rightarrow \infty} \Phi_T(Y)= (\delta_0 \otimes A_0).
\end{align*}
Define now 
\begin{align*}
 \Phi_0(Y)=\lim\limits_{T\rightarrow \infty} \Phi_T(Y)= \lim\limits_{T\rightarrow \infty}
\frac{1}{2T}\int_{-T}^{T} \phi_{e^{it}}(Y) dt.
\end{align*}
 Since $\|\Phi_T(Y)\|\leq \|Y\|$ for all $T>0$, it follows that $\|\Phi_0(Y)\|\leq \|Y\|$, for every generalized trigonometric polynomial  $Y$. So $\Phi_0$ can be extended to a linear contraction in $\mathcal{A}\times_\alpha \mathbb{R}_d$. In addition, since the family of operators $\{\Phi_T\,:\, T>0\}$ is uniformly bounded, applying a simple approximation argument, it follows that $\Phi_0(X)=\lim\limits_{T\rightarrow \infty} \Phi_T(X)$. This proves the following result.

\begin{prop}\label{expeker}
Let $E_0$ be the expectation defined in Theorem \ref{conexp} and $X\in\mathcal{A}\times_\alpha \mathbb{R}_d$. Then $\Phi_0(X)=\lim\limits_{T\rightarrow \infty}
\frac{1}{2T}\int_{-T}^{T} \phi_{e^{it}}(X) dt =\iota(E_0(X))$. 
\end{prop}

Applying standard arguments for kernels of approximating polynomials (cf. \cite{bes, lev-zhi}), we can  obtain the analogue of Bochner - Fejer's theorem.

Given a rationally independent set $\{\beta_1,\dots,\beta_m\}$ of real numbers and $X\in \mathcal{A}\times_\alpha \mathbb{R}_d$, one can define the \textbf{Bochner-Fejer polynomial} 
\begin{equation}\label{CesaroBF}
\sigma_{(\beta_1,\dots,\beta_m)}(X)=
\sum_{\substack{|\nu_1|<(m!)^2\\.........\\|\nu_m|<(m!)^2}}
\left(1-\frac{|\nu_1|}{(m!)^2}\right)\dots \left(1-\frac{|\nu_m|}{(m!)^2}\right)
(\delta_{\frac{\nu_1}{m!}\beta_1+\dots+ \frac{\nu_m}{m!}\beta_m}\otimes E_{{\frac{\nu_1}{m!}\beta_1+\dots+ \frac{\nu_m}{m!}\beta_m}}(X)).
\end{equation}
Note that a term of $\sigma_{(\beta_1, \dots,\beta_m)}(X)$ in \eqref{CesaroBF} differs from zero if and only if the respective Fourier coefficient of the term is nonzero.

\begin{prop}
\begin{align*}
\sigma_{(\beta_1, \dots,\beta_m)}(X)= \lim_{T\rightarrow \infty} \frac{1}{2T}\int_{-T}^{T} \phi_{e^{it}}(X) (\delta_0\otimes K_{(\beta_1, \dots,\beta_m)}(t))dt,
\end{align*}
where $K_{(\beta_1, \dots,\beta_m)}$ is the Bochner - Fejer kernel for almost periodic functions.
\end{prop}
\begin{proof}
Fix $n$ and compute
\begin{align*}
&\sigma_{(\beta_1,\dots,\beta_m)}(X)=
\sum_{\substack{|\nu_1|<(m!)^2\\.........\\|\nu_m|<(m!)^2}}
\left(1-\frac{|\nu_1|}{(m!)^2}\right)\dots \left(1-\frac{|\nu_m|}{(m!)^2}\right)
(\delta_{\frac{\nu_1}{m!}\beta_1+\dots+ \frac{\nu_m}{m!}\beta_m}\otimes E_{\frac{\nu_1}{m!}\beta_1+\dots+ \frac{\nu_m}{m!}\beta_m}(X))=\\
&=\sum_{\substack{|\nu_1|<(m!)^2\\.........\\|\nu_m|<(m!)^2}}
\left(1-\frac{|\nu_1|}{(m!)^2}\right)\dots \left(1-\frac{|\nu_m|}{(m!)^2}\right)
(\delta_0\otimes E_0(X(\delta_{-\frac{\nu_1}{m!}\beta_1-\dots- \frac{\nu_m}{m!}\beta_m}\otimes 1)))
(\delta_{\frac{\nu_1}{m!}\beta_1+\dots+ \frac{\nu_m}{m!}\beta_m}\otimes 1)=\\&=
\sum_{\substack{|\nu_1|<(m!)^2\\.........\\|\nu_m|<(m!)^2}}
\left(1-\frac{|\nu_1|}{(m!)^2}\right)\dots \left(1-\frac{|\nu_m|}{(m!)^2}\right)
\Phi_0(X(\delta_{-\frac{\nu_1}{m!}\beta_1-\dots- \frac{\nu_m}{m!}\beta_m}\otimes 1))
(\delta_{\frac{\nu_1}{m!}\beta_1+\dots+ \frac{\nu_m}{m!}\beta_m}\otimes 1)=\\&=
\lim_{T\rightarrow\infty}\frac{1}{2T} \int_{-T}^{T}
\sum_{\substack{|\nu_1|<(m!)^2\\.........\\|\nu_m|<(m!)^2}}
\left(1-\frac{|\nu_1|}{(m!)^2}\right)\dots \left(1-\frac{|\nu_m|}{(m!)^2}\right)
\phi_{e^{it}}(X) (\delta_0\otimes e^{-it({\frac{\nu_1}{m!}\beta_1+\dots+ \frac{\nu_m}{m!}\beta_m})}\cdot 1) dt=\\&=
\lim_{T\rightarrow\infty}\frac{1}{2T} \int_{-T}^{T}
\phi_{e^{it}}(X) (\delta_0\otimes K_{(\beta_1,\dots,\beta_m)}(t)\cdot 1) dt.
\end{align*}
\end{proof}

\begin{cor}\label{conces}
For every finite rationally independent set $\{\beta_1,\dots,\beta_m\}$, the map $\sigma_{(\beta_1,\dots,\beta_m)}$ is contractive.
\end{cor}
Let $X\in \mathcal{A}\times_\alpha \mathbb{R}_d$. By the previous lemma we have
\begin{align*}
 \|\sigma_{(\beta_1,\dots,\beta_m)}(X)\|&\leq  \lim_{T\rightarrow\infty}\frac{1}{2T} \int_{-T}^{T}
\|\phi_{e^{it}}(X)\| \|(\delta_0\otimes K_{(\beta_1,\dots,\beta_m)}(t)\cdot 1)\| dt=\\&=
 \lim_{T\rightarrow\infty}\frac{1}{2T} \int_{-T}^{T}
 K_{(\beta_1,\dots,\beta_m)}(t) dt \|X\|=\|X\|.
\end{align*}

Let now $X\in \mathcal{A}\times_\alpha \mathbb{R}_d$ and $(Y_n)_n$ be a sequence of generalized trigonometric polynomials in $\mathcal{A}\times_\alpha \mathbb{R}_d$ that converges to $X$. Define $\Gamma=\cup_n \Gamma_{n}$ as in the proof of Proposition \ref{countnumfoco}
and  let $B=(\beta_1,\beta_2,\dots,\beta_m,\dots)$ be a rational basis of $\Gamma$.

\begin{thm}\label{Cesaro covergence} 
$\sigma_{(\beta_1,\dots,\beta_m)}(X)\stackrel{\|\cdot\|}{\rightarrow}X$, as $M\rightarrow \infty$.
\end{thm}
\begin{proof}
 We will show first that $\sigma_{(\beta_1,\dots,\beta_m)}(Y_n)\stackrel{\|\cdot\|}{\rightarrow}Y_n$, for every $n\in\mathbb{N}$. 
Fix some $n\in\mathbb{N}$. Suppose that $Y_n$ is the trigonometric polynomial $\sum\limits_{\substack{s\in F\\F\subset\subset\mathbb{R}}}(\delta_s\otimes A_s)$. Since $B$ is also a rational basis of the indices of the nonzero Fourier coefficients of $Y_n$ we have
\begin{align*}
&\sigma_{(\beta_1,\dots,\beta_m)}(Y_n)= \\&=
\lim_{T\rightarrow\infty}\frac{1}{2T} \int_{-T}^{T}
\phi_{e^{it}}(Y_n) (\delta_0\otimes K_{(\beta_1,\dots,\beta_m)}(t)\cdot 1) dt=\\&=
\lim_{T\rightarrow\infty}\frac{1}{2T}\int_{-T}^T
\sum_{\substack{s\in F\\F\subset\subset\mathbb{R}}}(\delta_s\otimes e^{its}A_s)
(\delta_0\otimes K_{(\beta_1,\dots,\beta_m)}(t)\cdot 1)dt=\\& =\sum_{\substack{s\in F\\F\subset\subset\mathbb{R}}}\left((\delta_s\otimes A_s)\left(\lim_{T\rightarrow\infty}\frac{1}{2T}\int_{-T}^T e^{its}(\delta_0\otimes K_{(\beta_1,\dots,\beta_m)}(t)\cdot 1)dt \right) \right)\\&\rightarrow
\sum_{\substack{s\in F\\F\subset\subset\mathbb{R}}}(\delta_s\otimes A_s) (\delta_0\otimes 1)
=\sum_{\substack{s\in F\\F\subset\subset\mathbb{R}}} (\delta_s\otimes A_s).
\end{align*}
Given now $\epsilon>0$, choose trigonometric polynomial $Y_{n_0}$ with $\|X-Y_{n_0}\|<\epsilon/3$. Then there exists $m_0\in\mathbb{N}$, such that $\|Y_{n_0}-\sigma_{(\beta_1,\dots,\beta_m)}(Y_{n_0})\|\leq \epsilon/3$, for every $m>m_0$. Hence, it follows from Corollary \ref{conces} that for all $m>m_0$ we have
\begin{align*}
 \|X-\sigma_{(\beta_1,\dots,\beta_m)}(X)\|&\leq \|X-Y_{n_0}\|+\|Y_{n_0}-\sigma_{(\beta_1,\dots,\beta_m)}(Y_{n_0})\|+\|\sigma_{(\beta_1,\dots,\beta_m)}(Y_{n_0}-X)\|\leq\\&
\leq 2\|X-Y_{n_0}\|+\|Y_{n_0}-\sigma_{(\beta_1,\dots,\beta_m)}(Y_{n_0})\|\leq \epsilon.
\end{align*}
\end{proof}
\begin{cor}\label{uniqfoco}
Let $X\in\mathcal{A}\times_\alpha \mathbb{R}_d$, such that $E_s(X)=0$, for every $s\in \mathbb{R}$. Then $X=0$.
\end{cor} 
\begin{proof}
Since $E_s(X)=0$, for every $s\in \mathbb{R}$, it follows that $\Phi_s(X)=\iota(E_s(X))=0$, for every $t\in \mathbb{R}$. Hence the Bochner-Fejer polynomials of $X$ are trivial, so by Theorem \ref{Cesaro covergence} we have $X=0$.
\end{proof}
Recall now the left regular representation of the C$^\ast$-dynamical system, given by the formulas \eqref{lrrep1} and \eqref{lrrep2}. As we stated in Remark \ref{reducedcp} the left regular representation gives rise to the reduced crossed product. The following result comes readily from the previous theorem.

\begin{prop}\label{redeqfull}
Let $(\mathcal{A},\mathbb{R}_d,\alpha)$ be a C$^\ast$-dynamical system. Then the reduced crossed product 
$\mathcal{A}\times_\alpha^r \mathbb{R}_d$ coincides with the full crossed product 
$\mathcal{A}\times_\alpha \mathbb{R}_d$.
\end{prop}
\begin{proof}
By the universal property of the full crossed product, there is a representation
\begin{align*}
\phi: \mathcal{A}\times_\alpha \mathbb{R}_d\rightarrow \mathcal{A}\times_\alpha^r \mathbb{R}_d :
(\delta_s\otimes A)\mapsto\tilde{\pi}(A)\Lambda_s.
\end{align*}
It suffices to show that $\phi$ is faithful. We need first to point out some observations about these two C$^\ast$-algebras. 

By Remark \ref{reducedcp}, one can define on 
$\mathcal{A}\times_\alpha^r \mathbb{R}_d$ the contractive maps
\begin{align*}
\tilde{E}_t: \mathcal{A}\times_\alpha^r \mathbb{R}_d\rightarrow A : 
\sum\limits_{\substack{s\in F \\ F \subset\subset \mathbb{R}}}\tilde{\pi}(A_s)\Lambda_s \rightarrow \tilde{\pi}(A_t).
\end{align*}
Let now $\{\Phi_t\,:\,t>0\}$ be the family of contractions on $\mathcal{A}\times_\alpha \mathbb{R}_d$, given by the formula
\begin{equation}\label{phitcont}
\Phi_t(X)=\Phi_0(X(\delta_{-t}\otimes 1)),
\end{equation} 
where $\Phi_0$ is the operator defined in Proposition \ref{expeker}. It follows by routine calculations 
on the subalgebra of trigonometric polynomials and standard density arguments that $\tilde{E}_t\circ \phi=\phi \circ \Phi_t$, for all $t\in\mathbb{R}$.

Let now $X\in \mathcal{A}\times_\alpha \mathbb{R}_d$, such that $\phi(X)=0$. Then $(\tilde{E}_t\circ \phi)(X)=0$
for every $t\in \mathbb{R}$, which implies that $(\phi \circ \Phi_t)(X)=0$. Since the left regular representation is a faithful representation of $\mathcal{A}\mathbb{R}_d$, it follows that $\Phi_t(X)=0$, for every $t\in \mathbb{R}$. Hence by Corollary \ref{uniqfoco} we have $X=0$.
\end{proof}
As a simple consequence of the above proposition we obtain the following useful inequality. Note that it essentially corresponds to the elementary fact that the norm of an $\mathbb{R}_d\times \mathbb{R}_d$ operator matrix $X$ dominitates the norm of any of its columns. 

\begin{prop}\label{estimforredcros}
Let $\mathcal{A}$ be a C$^\ast$-algebra acting on a Hilbert space $H$ and $\xi$ be a unit vector in $H$.
For every $X\in \mathcal{A}\times_\alpha\mathbb{R}_d$ and $F$ arbitrarily chosen finite subset of $\mathbb{R}$, we have
\begin{align*}
\| (\tilde{id}\rtimes \Lambda)(X)\|^2 -\sum\limits_{\substack{s\in F \\ F \subset\subset \mathbb{R}}}\|\alpha_{-s}(E_s(X))\xi\|^2\geq 0.
\end{align*}
\end{prop}
\begin{proof}
Applying Proposition \ref{redeqfull}, it suffices to prove the result for the reduced crossed product norm.  Let $\operatorname{id}$ be the identity representation of $\mathcal{A}$ on $H$ and $Y=\sum\limits_{\substack{s\in F \\ F \subset\subset \mathbb{R}}}(\delta_s\otimes A_s)$ be a generalized trigonometric polynomial in $\mathcal{A}\times_\alpha\mathbb{R}_d$. Note first that  
 \begin{align*}
&\bigg\|\big((\tilde{\operatorname{id}}\rtimes \Lambda)(X)- \sum_{\substack{s\in F \\ F \subset\subset \mathbb{R}}}(\tilde{\operatorname{id}}\rtimes \Lambda)(\delta_s\otimes A_s)\big)
\delta_{0,\xi}\bigg\|^2 
=\| (\tilde{\operatorname{id}}\rtimes \Lambda)(X)\delta_{0,\xi}\|^2+
\bigg\|\sum\limits_{\substack{s\in F \\ F \subset\subset \mathbb{R}}}\tilde{\operatorname{id}}(A_s)\delta_{s,\xi}\bigg\|^2 \\
&\,\,\,\,\,\,\,\,\,\,\,\,\,\,\,\,\,\,\,\,\,\,\,\,\,\,\,\,\,\,\,\,\,\,\,- \sum\limits_{\substack{s\in F \\ F \subset\subset \mathbb{R}}}\langle (\tilde{\operatorname{id}}\rtimes \Lambda)(X)\delta_{0,\xi}, 
\tilde{\operatorname{id}}(A_s)\delta_{s,\xi}\rangle - 
\sum\limits_{\substack{s\in F \\ F \subset\subset \mathbb{R}}}\langle  
\tilde{\operatorname{id}}(A_s)\delta_{s,\xi}, (\tilde{\operatorname{id}}\rtimes \Lambda)(X)\delta_{0,\xi}\rangle
\end{align*}
Since $\tilde{\operatorname{id}}(A_s)\delta_{s,\xi}= \delta_{s,\alpha_{-s}(A_s)\xi}=V_s (\alpha_{-s}(A_s)\xi)$ and $\delta_{s,\xi}$ is orthogonal to $\delta_{t,\eta}$ for $s\neq t$, it follows that  
\begin{align*}
&\bigg\|\big((\tilde{\operatorname{id}}\rtimes \Lambda)(X)- \sum_{\substack{s\in F \\ F \subset\subset \mathbb{R}}}(\tilde{\operatorname{id}}\rtimes \Lambda)(\delta_s\otimes A_s)\big)
\delta_{0,\xi}\bigg\|^2 
=
\| (\tilde{\operatorname{id}}\rtimes \Lambda)(X)\delta_{0,\xi}\|^2+
\sum\limits_{\substack{s\in F \\ F \subset\subset \mathbb{R}}}\|\alpha_{-s}(A_s)\xi\|^2 
\\&\,\,\,\,\,\,\,\,\,\,\,\,\,\,\,\,\,\,\,\,\,\,\,\,\,\,\,\,\,\,\,\,\,\,\,
-\sum\limits_{\substack{s\in F \\ F \subset\subset \mathbb{R}}}\langle V_s^\ast(\tilde{\operatorname{id}}\rtimes \Lambda)(X)V_0\xi, 
\alpha_{-s}(A_s)\xi\rangle -
\sum\limits_{\substack{s\in F \\ F \subset\subset \mathbb{R}}}\langle  
\alpha_{-s}(A_s)\xi, V_s^\ast(\tilde{\operatorname{id}}\rtimes \Lambda)(X)V_0\xi \rangle.
\end{align*}
One may check that  $V_s^\ast(\tilde{\operatorname{id}}\rtimes \Lambda)(X)V_0=\alpha_{-s}(E_s(X))$, so adding and subtracting $\sum\limits_{\substack{s\in F \\ F \subset\subset \mathbb{R}}}\|\alpha_{-s}(E_s(X))\xi\|^2$, we obtain that the above expression is equal to  
\begin{align*}
 \| (\tilde{\operatorname{id}}\rtimes \Lambda)(X)\delta_{0,\xi}\|^2 -\sum\limits_{\substack{s\in F \\ F \subset\subset \mathbb{R}}}\|\alpha_{-s}(E_s(X))\xi\|^2+
\sum\limits_{\substack{s\in F \\ F \subset\subset \mathbb{R}}}\|\alpha_{-s}(E_s(X))\xi-\alpha_{-s}(A_s)\xi\|^2.
\end{align*}
Note that the last formula takes its lowest value when $\sum\limits_{\substack{s\in F \\ F \subset\subset \mathbb{R}}}\|\alpha_{-s}(E_s(X))\xi-\alpha_{-s}(A_s)\xi\|^2=0$, which happens in the case we choose $A_s=E_s(X)$. Since the left hand side is non-negative, we deduce that
\begin{align*}
\| (\tilde{\operatorname{id}}\rtimes \Lambda)(X)\delta_{0,\xi}\|^2 -\sum\limits_{\substack{s\in F \\ F \subset\subset \mathbb{R}}}\|\alpha_{-s}(E_s(X))\xi\|^2\geq 0.
\end{align*}
\end{proof}

\subsection{Semicrossed products}
\begin{definition}
Let $(\mathcal{A}, G,\alpha)$ be a C$^\ast$-dynamical system. If $\mathcal{B}$ is a unital closed subalgebra of $\mathcal{A}$ and $G^+$ is a unital semigroup of $G$, we define the \textbf{semicrossed product} $\mathcal{B}\times_\alpha G^+$ as the closed subalgebra of the full crossed product, that is generated by the elements $(\delta_0\otimes b), (\delta_s\otimes 1)$, with $b\in \mathcal{B}$ and $s\in G^+$.
\end{definition}  
\begin{prop}\label{charwithfoucoe}
Let $(\mathcal{A}, \mathbb{R}_d,\alpha)$ be a C$^\ast$-dynamical system. Then the semicrossed product $\mathcal{A}\times_\alpha \mathbb{R}_d^+$ is equal to the set
\begin{align*}
\mathcal{A}^{\mathbb{R}^+}=\{X\in \mathcal{A}\times_\alpha \mathbb{R}_d : E_s(X)=0, \text{ for all }s<0\}.
\end{align*} 
\end{prop}
\begin{proof}
If $X$ is a trigonometric polynomial in $\mathcal{A}\times_\alpha \mathbb{R}_d^+$, then it is trivial to see that $X$ lies in $\mathcal{A}^{\mathbb{R}^+}$. The latter set is closed, since it is the intersection of the kernels $\operatorname{ker} E_s$ for all $s>0$, so the first inclusion is proved. For the converse inclusion, suppose $X\in \mathcal{A}^{\mathbb{R}^+}$. If $X=0$, there is nothing to prove. If $X\neq 0$, then the only nonzero Fourier coefficients of $X$  have nonnegative indices, so the Fejer-Bochner polynomials of $X$ lie in  $\mathcal{A}\times_\alpha \mathbb{R}_d^+$. Hence by Theorem \ref{Cesaro covergence} we have that $X\in \mathcal{A}\times_\alpha \mathbb{R}_d^+$.
\end{proof}
The following corollary follows trivially by routine calculations on the generalized trigonometric polynomials of the semicrossed product algebra.
\begin{cor}\label{contrexpmult}
Let $(\mathcal{A}, \mathbb{R}_d,\alpha)$ be a C$^\ast$-dynamical system. The restriction of the expectation $E_0$ to $\mathcal{A}\times_\alpha \mathbb{R}_d^+$ is a contractive homomorphism onto $\mathcal{A}$.
\end{cor}

\subsection{An example : the algebra of analytic almost periodic functions}
The theory of almost periodic functions was mainly created in 1925 by Bohr \cite{boh} and was substantially developed during the 1930s by Bochner, Besicovich, Stepanov, amongst others. The reader can refer to \cite{bes, lev-zhi} for more details. 

We recall that a continuous function $f:\mathbb{R}\rightarrow \mathbb{C}$ is almost periodic if and only if for every $\epsilon>0$ the set of $\epsilon$-translation numbers\footnote{A set $E\subseteq\mathbb{R}$ is called relatively dense if there exists $\lambda > 0$ such
that any interval of length $\lambda$ contains at least one element of $E$.} is relatively dense\footnote{Given a function $f:\mathbb{R}\rightarrow \mathbb{C}$
 and $\epsilon>0$, a real number $\tau$ is called an $\epsilon$-translation number of $f$, if
\begin{align*}
\sup_{t\in\mathbb{R}} |f(t + \tau) - f(t)| \leq \epsilon.
\end{align*}} 
in $\mathbb{R}$.  We denote by $AP(\mathbb{R})$ the algebra of almost periodic functions and we equip it with the supremum norm. Using a standard approximation argument, one can check that $AP(\mathbb{R})$ is a norm closed selfadjoint subalgebra of $C_b(\mathbb{R})$. Equivalently, we can define $AP(\mathbb{R})$ as the closure of the set of trigonometric polynomials 
\begin{align*}
p(x) =\sum_{k=1}^n c_k e^{i\lambda_k x},\text{ with }c_k \in \mathbb{C}, \lambda_k\in\mathbb{R}.
\end{align*}

We shall show that $AP(\mathbb{R})$ is isomorphic as a C$^\ast$-algebra to the crossed product $\mathbb{C}\times \mathbb{R}_d$ (with the trivial action). First we need the following well known result from elementary Gelfand theory. We include its proof for completeness. 

\begin{prop}\label{cgisom} Let $G$ be a discrete abelian group. The crossed product $\mathbb{C}\times G$ is isometrically isomorphic to the C$^\ast$-algebra $C(\hat{G})$ of continuous functions on the dual group $\hat{G}$.
\end{prop}
\begin{proof}
The crossed product $\mathbb{C}\times G$ is a unital commutative algebra, so by the Gelfand transform it is isometrically isomorphic with $C(\mathfrak{M}(\mathbb{C}\times G))$. Hence it suffices to identify the character space $\mathfrak{M}(\mathbb{C}\times G)$ with $\hat{G}$.

If $\chi\in\mathfrak{M}(\mathbb{C}\times G)$, we may restrict to $G\subset \mathbb{C}G\subseteq \mathbb{C}\times G$. Since $\chi$ is multiplicative, $\chi\big|_G$ is a group homomorphism. Define
\begin{align*}  
\alpha: \mathfrak{M}(\mathbb{C}\times G)\rightarrow \hat{G}:\chi\mapsto \chi\big|_G.
\end{align*}
\begin{itemize}
	\item $\alpha$ is injective; indeed, if $\chi\big|_G=\phi\big|_G$ then it follows by linearity that $\chi\big|_{\mathbb{C}G}=\phi\big|_{\mathbb{C}G}$. Hence  $\chi=\phi$, since they are continuous on $\mathfrak{M}(\mathbb{C}\times G)$ and they coincide on a dense subset.
	\item  $\alpha$ is surjective; given $\chi\in \hat{G}$, define the $\ast$-homomorphism
	\begin{align*}
	\pi_\chi: \mathbb{C}G\rightarrow \mathbb{C}: \sum\limits_{\substack{s\in F \\ F \subset\subset G}}(\delta_g\otimes a_g)\mapsto
	 \sum\limits_{\substack{s\in F \\ F \subset\subset G}} \chi(g) a_g.
	\end{align*}
	Then  
	\begin{align*}
	\bigg\|\pi_\chi\left(\sum\limits_{\substack{s\in F \\ F \subset\subset G}}(\delta_g\otimes a_g)\right)\bigg\|&=
	\big\|\sum\limits_{\substack{s\in F \\ F \subset\subset G}} \chi(g) a_g\big\|\leq 
	\sum\limits_{\substack{s\in F \\ F \subset\subset G}} \|\chi(g)\|\,\|a_g\|=\\&=
	\sum\limits_{\substack{s\in F \\ F \subset\subset G}} \|a_g\|=\big\| \sum\limits_{\substack{s\in F \\ F \subset\subset G}}(\delta_g\otimes a_g)\big\|_{\ell_1}.
	\end{align*}
	So by the universal property of crossed products, we can extend $\pi_\chi$ to a nonzero representation of $\mathbb{C}\times G$ on $\mathbb{C}$.
	\item $\alpha$ is evidently continuous. Since its domain is a compact space, $\alpha$ is a homeomorphism.
\end{itemize}
\end{proof}

Set now $G$ equal to $\mathbb{R}_d$. Applying the above proposition we obtain that $\mathbb{C}\times \mathbb{R}_d$ is isometrically isomorphic with $C(\mathbb{R}_B)$, where $\mathbb{R}_B$ is the Bohr compactification of the real numbers. The latter C$^\ast$-algebra can be identified with $AP(\mathbb{R})$ \cite{shu}. 
In the next proposition we provide a proof, using the machinery of crossed products.
\begin{prop}
The commutative C$^\ast$-algebras $AP(\mathbb{R})$ and $C(\mathbb{R}_B)$ are isomorphic.
\end{prop}
\begin{proof}
By Proposition \ref{cgisom}, we identify $C(\mathbb{R}_B)$ with the crossed product $\mathbb{C}\times \mathbb{R}_d$.
Define the covariant representation $(\pi, U)$ of the C$^\ast$-dynamical system $(\mathbb{C},\mathbb{R}_d,\operatorname{id})$ by the formulas
\begin{align*}
\mathbb{C}\rightarrow AP(\mathbb{R}): c\mapsto c\cdot 1
\end{align*}
and
\begin{align*}
\mathbb{R}\rightarrow AP(\mathbb{R}): \lambda \mapsto e^{i\lambda x}.
\end{align*}
By the universal property of crossed products, we obtain a representation $\tilde{\pi}$  given by 
\begin{align*}
\tilde{\pi}:\mathbb{C}\times \mathbb{R}_d\rightarrow AP(\mathbb{R}):\sum\limits_{\substack{s\in F \\ F \subset\subset \mathbb{R}}}(\delta_s\otimes a_s)\mapsto
\sum\limits_{\substack{s\in F \\ F \subset\subset \mathbb{R}}}a_s\,e^{isx}.
\end{align*}
Let $X\in \mathbb{C}\times \mathbb{R}_d$, such that  $\tilde{\pi}(X)=0$. 
One can check that, as in the proof of Proposition \ref{redeqfull} that $(\tilde{\pi}\circ E_\lambda) (X)=(\epsilon_\lambda \circ \tilde{\pi})(X)$, where $\epsilon_\lambda$ is given by the formula 
\begin{align*}
\epsilon_\lambda(f) = \lim_{T\rightarrow\infty}\frac{1}{2T}\int_{-T}^{T} f(t)e^{-i\lambda t}dt.
\end{align*}
Hence it follows that $E_\lambda (X)=0$, for every $\lambda\in \mathbb{R}$, so we have by Theorem \ref{Cesaro covergence} that $X=0$.
Thus, $\tilde{\pi}$ is injective and the proof is complete.
\end{proof}

We focus now on the non-selfadjoint algebra of analytic almost periodic functions, that is the norm closed algebra generated by the functions $\{e^{i\lambda x} : \lambda \geq 0\}$.  Applying Proposition \ref{charwithfoucoe} and Corollary \ref{contrexpmult}, we have the following result.
\begin{prop}\label{charxinfty}
\begin{align*}
AAP(\mathbb{R})=\{f\in AP(\mathbb{R}): \epsilon_\lambda(f) =0 , \text{ for every }\lambda<0\}.
\end{align*}
Moreover, the compression of the contractive map $\epsilon_0$ to $AAP(\mathbb{R})$ is multiplicative; hence it induces a character 
$x_\infty$ that satisfies
\begin{align*} 
x_\infty\left(\sum
_{\substack{\lambda\in F \\ F\subset\subset \mathbb{R}^+}} 
c_\lambda e^{i\lambda x}\right)\mapsto c_0.
\end{align*}
\end{prop}

Let now $\phi_{c,k}$ be the continuous multiplicative linear map, given by the formula 
\begin{align*}
\phi_{c,k}(e^{i\lambda x}) = c(\lambda)e^{ik\lambda x},
\end{align*}
where $k > 0$ and $c : \mathbb{R}\mapsto \mathbb{T}$ homomorphism (so $c \in \mathbb{R}_B$). 
Arens proved in \cite{are} that the set of continuous automorphisms of $AAP(\mathbb{R})$ is equal to the set $\{\phi_{c,k}: c\in\mathbb{R}_B,  k\in\mathbb{R}^+\}$.

\begin{prop}\label{autapisomisom}
Every automorphism $\phi_{c,k}$ is isometric.
\end{prop}
\begin{proof}
Fix some $c\in\mathbb{R}_B$ and $k\in\mathbb{R}^+$. One can check that $(\tilde{\operatorname{id}}, u_{c,k})$, where $\tilde{\operatorname{id}}:\mathbb{C}\rightarrow AP(\mathbb{R}): c \mapsto c\cdot 1_\mathbb{R}$ and $u_{c,k}:\mathbb{R}_d\rightarrow AP(\mathbb{R}):\lambda\mapsto c(\lambda)e^{ik\lambda  x}$, gives a covariant representation of the C$^\ast$-dynamical system $(\mathbb{C},\mathbb{R}_d,\operatorname{id})$. Hence by the universal property, we have a representation of the $C^\ast$-algebra  $\mathbb{C}\times \mathbb{R}_d\simeq C(\mathbb{R}_B)\simeq AP(\mathbb{R})$ of almost periodic functions, given by 
\begin{equation}\label{eqautapbab}
\tilde{\operatorname{id}}\rtimes u_{c,k}: AP(\mathbb{R})\rightarrow AP(\mathbb{R}): e^{i\lambda x} \mapsto  c(\lambda) e^{ik\lambda x}.
\end{equation}
The representation $\tilde{\operatorname{id}}\rtimes u_{c,k}$ is evidently faithful, so it is isometric. Moreover, the restriction of $\tilde{\operatorname{id}}\rtimes u_{c,k}$  to the invariant subalgebra $AAP(\mathbb{R})$ is equal to $\phi_{c,k}$, so the proof is complete.
\end{proof}

\section{The norm closed parabolic algebra $A_p$}

Let $(AP(\mathbb{R}),\mathbb{R}_d,\tau)$ be a C$^\ast$-dynamical system, where $\tau$ induces the group of translation automorphisms:
\begin{align*}
(\tau_s f)(x)=f(x-s), \,\, f\in AP(\mathbb{R}).
\end{align*}
Our goal in this section, is to prove that the abstract discrete crossed product $AP(\mathbb{R})\times_\tau \mathbb{R}_d$ is isometrically isomorphic to a concrete C$^\ast$-algebra acting on $L^2(\mathbb{R})$. 
\begin{prop}
The crossed product $AP(\mathbb{R})\times_\tau \mathbb{R}_d$ is a simple algebra, i.e. it has no non-trivial two-sided closed ideals.
\end{prop}
\begin{proof}
Let $J$ be a non-zero two-sided closed ideal. Hence there exists an element $X\in J$, such that $\Phi_s(X)\neq 0$, for some $s\in\mathbb{R}$.  Using the integral formula $\Phi_0(X)=\lim\limits_{T\rightarrow +\infty}\frac{1}{2T}\int_{-T}^{T} \phi_{e^{it}}(X)dt$ that we proved in the previous section, we will prove that $\Phi_s(X)$ belongs to $J$.
Since $J$ is closed, it suffices to prove that $\phi_{e^{it}}(X)\in J$. Suppose first that $X$ is a generalized trigonometric polynomial 
$\sum\limits_{\substack{s\in F\\F\subset\subset\mathbb{R}}}(\delta_s\otimes f_s)$.
Compute the product
 \begin{align*}
 (\delta_0\otimes e^{itx})X(\delta_0\otimes e^{-itx})&=
 \sum\limits_{\substack{s\in F\\F\subset\subset\mathbb{R}}}(\delta_s\otimes e^{itx}f_s)(\delta_0\otimes e^{-itx})=\\&=
 \sum\limits_{\substack{s\in F\\F\subset\subset\mathbb{R}}}(\delta_s\otimes e^{itx}f_s   \tau_s(e^{-itx}))=\\&=
 \sum\limits_{\substack{s\in F\\F\subset\subset\mathbb{R}}}(\delta_s\otimes e^{itx}f_s e^{-it(x-s)})=\\&=
 \sum\limits_{\substack{s\in F\\F\subset\subset\mathbb{R}}}(\delta_s\otimes e^{its}f_s)=
 \phi_{e^{it}}(X).
 \end{align*}
Hence, it follows by a standard approximation argument that $(\delta_0\otimes e^{itx})X(\delta_0\otimes e^{-itx})=  \phi_{e^{it}}(X)$ for any $X\in  AP(\mathbb{R})\times_\tau \mathbb{R}_d$.

 Similarly, we get $\Phi_s(X)=\Phi(X(\delta_{-s}\otimes 1))\in J$, so there exists some nonzero $f\in AP(\mathbb{R})$, such that $(\delta_0\otimes f)\in J$.
Since the action of the group can be described by the product of the covariant relation, it follows $(\delta_0\otimes \tau_s(f))\in J$ for every $s\in\mathbb{R}$.

\textbf{Claim:} We may assume that $\inf\{|f(x)|\,:\, x\in\mathbb{R}\}\geq c >0$.

Since $f\cdot f^\ast, nf\in AP(\mathbb{R})$ for every $n\in \mathbb{N}$, we may assume 
that $f(x)\geq0$, for every $x\in\mathbb{R}$ and $\|f\|>2$. Let 
$\epsilon=\frac{1}{2}$. Then there is $T=T(\epsilon)>0$, such that for every interval $I$ of length $T$, there exists $\ell \in I$ that satisfies
\begin{align*}
|f(x+\ell)-f(x)|<\epsilon, \forall x\in \mathbb{R}.
\end{align*}
On the interval $[0,T]$, we may assume that $f(x)>1$, for every $x\in[0,\frac{T}{n}]$, for some $n\in\mathbb{N}$ (otherwise, work with $g=\tau_s(f)$, for suitable $s$). Then, let
$f_k=\tau_{k\frac{T}{n}}(f)$, for $k=0,1,\dots,n-1$ and define 
\begin{align*}
g(x)=\sum_{k=0}^{n-1}f_{k}(x), \,\,x\in\mathbb{R}.
\end{align*}
Then $g(x)>1$, for every $x\in[0,T]$. In the general case where $x\in\mathbb{R}$, there exists $\ell\in[x-T,x]$, such that $|f(x-\ell)-f(x)|<\epsilon$. Since $\ell$ gives that bound uniformly for all $y\in\mathbb{R}$, it yields that $|f_k(x-\ell)-f_k(x)|<\epsilon$, for every $k\in\{0,1,\dots,n-1\}$. Therefore, there exists some $k$, such that $|f_k(x)|>1-\epsilon=\frac{1}{2}$. Hence $g(x)>\frac{1}{2}$ and that completes the proof of our claim.

 Now, since the value $\inf\{|f(x)|: x\in\mathbb{R}\}$ is positive, the multiplicative inverse $1/f$ is a bounded almost periodic function \cite{bes}. Then
\begin{align*}
(\delta_0\otimes f)(\delta_0\otimes 1/f)=(\delta_0\otimes 1)\in I,
\end{align*}  
so $I$ coincides with the crossed product.
\end{proof}
\begin{rem}
The simplicity of crossed product algebras has been studied extensively over the last 50 years (see for example \cite{eff-hah, har-ska}). In particular, Archbold and Spielberg proved in \cite{arc-spi} that given a C$^\ast$-dynamical system $(\mathcal{A},G,\alpha)$, with $\mathcal{A}$ commutative and $G$ discrete, the crossed product $\mathcal{A}\times_\alpha G$ is simple if and only if the action of the group on $\mathcal{A}$ is minimal\footnote{The action of a group $G$ on a C$^\ast$-algebra $\mathcal{A}$ is called minimal if $\mathcal{A}$ does not contain any non-trivial G-invariant ideals.} and topologically free\footnote{An action $\alpha$ on a commutative algebra $\mathcal{A}$ is said to be topologically free if for any finite set $F\subseteq G\backslash\{e_G\}$, the set  $\cap_{t\in F}\{\chi\in \mathfrak{M}(\mathcal{A})|\chi\circ \alpha_t\neq \chi\}$ is dense in $\mathfrak{M}(\mathcal{A})$.}.
\end{rem}
\begin{definition}
Let $B_p$ be the C$^\ast$-algebra that is generated by the set of all the multiplication and translation operators $M_\lambda$ and $D_\mu$ acting on $L^2(\mathbb{R})$ respectively. 
Since the span of the products $M_\lambda D_\mu$ is closed under the operations of ring multiplication and involution, we get that 
\begin{align*}
B_p=\overline{\textrm{span}\{M_\lambda D_\mu\,:\,\lambda, \mu\in\mathbb{R}\}}^{\|\cdot\|}.
\end{align*}
\end{definition}
\begin{thm}
The C$^\ast$-algebras $AP(\mathbb{R})\times_\tau \mathbb{R}_d$ and $B_p$ are isomorphic.
\end{thm}
\begin{proof}
Define the covariant representation $(\pi, D)$, where:
\begin{align*}
\pi: AP(\mathbb{R})\rightarrow B(L^2(\mathbb{R})): e^{i\lambda x}\mapsto M_\lambda
\end{align*}
and
\begin{align*}
D: \mathbb{R}_d\rightarrow B(L^2(\mathbb{R})): \mu\mapsto D_\mu.
\end{align*}
It is trivial to see that $\pi$ is a representation of $AP(\mathbb{R})$ and $D$ is a unitary representation, so it suffices to prove the covariance relation. Compute
\begin{align*}
D_\mu\pi(e^{i\lambda x})D_\mu^\ast=&D_\mu M_\lambda D_\mu^\ast
\end{align*}
and
\begin{align*}
\pi(\tau_\mu(e^{i\lambda x}))=&\pi(e^{i\lambda (x-\mu)})=e^{-i\lambda\mu}\pi(e^{i\lambda x})=e^{-i\lambda\mu}M_\lambda,
\end{align*}
hence the covariant relation holds by the Weyl relations.
By the universal property of the crossed product, this yields a representation between two C$^\ast$-algebras
\begin{align*}
(\pi\rtimes D): AP(\mathbb{R})\times_\tau \mathbb{R}_d\rightarrow C^\ast(\pi(AP(\mathbb{R})),D(\mathbb{R}_d))\,:\, (\pi\rtimes D) 
\sum_{\substack{s\in F\\F\subset\subset\mathbb{R}}}(\delta_s\otimes f_s)\mapsto 
\sum\limits_{\substack{s\in F\\F\subset\subset\mathbb{R}}}\pi(f_s)D_s.
\end{align*}
Observe that C$^\ast(\pi(AP(\mathbb{R})),D(\mathbb{R}_d))=B_p$.  Since 
$AP(\mathbb{R})\times_\tau \mathbb{R}_d$ is a simple algebra and $ker(\pi\rtimes D)$ is a two sided ideal, $(\pi\rtimes D)$ is injective, which yields that it is an isometric $\ast$-isomorphism.
\end{proof}

\begin{rem}
By the general theory of crossed products, the mapping
\begin{align*}
E_n: B_p\rightarrow B_p: \sum\limits_{\substack{s\in F\\F\subset\subset\mathbb{R}}} \pi(f_s)D_{\mu_s}\mapsto M_{f_n}
\end{align*}
is contractive. Moreover, we have a similar expectation for the $D_\mu$ operators. By the Weyl relations, we have the covariant relation $(\rho,M)$
\begin{align*}
\rho: AP(\mathbb{R})\rightarrow B(L^2(\mathbb{R})): e^{i\lambda x}\mapsto D_\lambda
\end{align*}
and
\begin{align*}
M: \mathbb{R}_d\rightarrow B(L^2(\mathbb{R})): \mu\mapsto M_{-\mu}.
\end{align*}
Hence, we have the isomorphism 
\begin{align*}
(\rho\rtimes M): AP(\mathbb{R})\times_\tau \mathbb{R}_d\rightarrow B_p\,:\, (\rho\rtimes M) 
\sum_{\substack{s\in F\\F\subset\subset\mathbb{R}}}(\delta_s\otimes f_s)\mapsto 
\sum\limits_{\substack{s\in F\\F\subset\subset\mathbb{R}}}\rho(f_s)M_{-s}.
\end{align*}
Therefore, we have the contractions
\begin{align*}
Z_m: B_p\rightarrow B_p: \sum\limits_{\substack{s\in F\\F\subset\subset\mathbb{R}}} \rho(f_s)M_{-\lambda_s}\mapsto D_{f_m}.
\end{align*}
Applying the natural isometric isomorphisms $M_f\mapsto f$ and $D_g\mapsto g$, we can identify the range of the maps $E_n$ and $Z_m$ with $AP(\mathbb{R})$.

One may check that $ (\rho\rtimes M)\circ(\pi\rtimes D)^{-1}\in \operatorname{Aut}(B_p)$, that sends $D_s$ to $M_{-s}$ and $M_\lambda$ to $D_\lambda$. Since $B_p$ is a concrete operator algebra on $L^2(\mathbb{R})$, by the Stone-von Neumann theorem   $(\rho\rtimes M)\circ(\pi\rtimes D)^{-1}=\operatorname{Ad}(F)$, 
where $\operatorname{Ad}(F)$ is the automorphism that is unitarily implemented by the Fourier-Plancherel transform \cite{mac}.
\end{rem}

The closed subalgebra of $B_p$ generated by $\{M_\lambda, D_\mu \,:\, \lambda,\mu\geq 0\}$ is called the \textbf{(norm closed) parabolic algebra} and it is denoted by $A_p$. Evidently,
\begin{align*}
(\pi\rtimes D)^{-1}(A_p)= AAP(\mathbb{R})\times_\tau \mathbb{R}_d^+,
\end{align*}
where $AAP(\mathbb{R})$ is the norm closed algebra of analytic almost periodic functions. Applying the contractions $E_n,\, Z_m$ we obtain by the standard Fejer-Bochner argument that 
\begin{align*}
AAP(\mathbb{R})\times_\tau \mathbb{R}_d^+=\{X\in AP(\mathbb{R})\times_\tau \mathbb{R}_d\,:\, E_n(X)=Z_m(X)=0, \text{ for all }n,m<0\}. 
\end{align*}
From now on, we identify $A_p$ with the semicrossed product $AAP(\mathbb{R})\times_\tau \mathbb{R}_d^+$. The first question to wonder for the norm closed algebra is once again the integral domain question, as in the weak$^\ast$-closed case. The question still seems hard to solve, because of the absence of a first nonzero coefficient. However we can prove that $A_p$ contains no non-trivial idempotents. The following lemma is the key.
\begin{prop}
The spectrum of every element $X$ in  $A_p$ is connected.
\end{prop}
\begin{proof}
Let $X\in AAP(\mathbb{R})\times_\tau \mathbb{R}_d^+$ with spectrum $Sp(X)=U\cup V$, where $U,V$ are non-empty disjoint compact subsets of $\mathbb{C}$. By the density of generalized trigonometric polynomials in $A_p$, there exists an element $X_0=\sum\limits_{\substack{s\in F\\F\subset\subset\mathbb{R}^+}} M_{g_s} D_s$, such that $Sp(X_0)$ is not connected (for this standard Banach algebra fact see for example Theorem 1.1 in \cite{her}).
Abusing the notation, we write again that $Sp(X_0)=U\cup V$, for some non-empty disjoint compact sets $U$ and $V$.
 
\noindent\underline{Claim:} The norm closed commutative algebra generated by a trigonometric polynomial $Z_0$, denoted by $A(Z_0)$, is an integral domain.\\
 Let $M>0$ and let $F_n$ be the finite set of positive indices of the nonzero Fourier coefficients of $Z_0^n$ (so $F_1=F\backslash\{0\}$). Since $Z_0$ has only a finite set of nonzero Fourier coefficients, there exists $N>0$, such that for every $n>N$ we have
\begin{align*}
F_n\cap [0,M]=\emptyset.
\end{align*}
Define $F_0=\cup_{n=1}^N F_n\cup \{0\}$. Then for every $t\in [0,M]\backslash F_0$ and $Y= \sum\limits_{n=0}^ N c_n Z_0^n$ generalized polynomial we have $E_t(Y)=0$. Since the subspace of generalized polynomials is dense in $A(Z_0)$ we obtain by continuity of the maps $E_t$ that
\begin{align*}
E_t(Y)=0, \text{ for all }Y\in A(Z_0).
\end{align*}
If $Y$ is a nonzero element in $A(Z_0)$, then it has some nonzero Fourier coefficient, say $E_{t_0}(Y)$. Hence the set of indices of nonzero Fourier coefficients of $Y$ in $[0,t_0]$ is finite and nonempty , so it follows that $Y$ has a first nonzero Fourier coefficient.

Let now $Y_1, Y_2$ be two nonzero elements of $A(Z_0)$ and let $m_1$ and $m_2$ be the indices of their respective first nonzero Fourier coefficients. Then $m_1+m_2$ is the first nonzero Fourier coefficient of the product $Y_1Y_2$, since
\begin{align*}
E_{m_1+m_2}(Y_1Y_2)=E_{m_1}(Y_1) \tau_{m_1} (E_{m_2}(Y_2)) 
\end{align*}
and $E_{m_1}(Y_1), \tau_{m_1} (E_{m_2}(Y_2))$ are two nonzero elements of the integral domain $AAP(\mathbb{R})$. Thus, we proved our claim.

 On the other hand, since $Sp(X_0)\subseteq U\cup V$, there are holomorphic functions $f,g$ defined on $U\cup V$, given by $f\big|_U=g\big|_V=1$ and $f\big|_V=g\big|_U=0$. Therefore it follows by Runge's theorem \cite{con} and the holomorphic functional calculus \cite{rad-ros} that $f(X_0),g(X_0)\in A(X_0)$ and 
\begin{align*}
 f(X_0)g(X_0)=0, 
\end{align*}
which contradicts the fact that $A(X_0)$ is an integral domain.
\end{proof}   
\begin{cor}
$A_p$ contains no non-trivial idempotents.
\end{cor}

\subsection{Isometric Automorphisms of $A_p$}
In this section, our goal is to determine the isometric automorphisms of the norm closed parabolic algebra. Interestingly there is a richer diversity than in the weak$^\ast$-span context. The automorphisms are strongly related to the characters of the discrete real line and the Arens - Singer theory for analytic almost periodic functions \cite{bks}. 

Recall that given a unitary map $U\in B(L^2(\mathbb{R}))$, we can define the automorphism
\begin{align*}
\operatorname{Ad}(U):B(L^2(\mathbb{R}))\rightarrow B(L^2(\mathbb{R})): T\mapsto UTU^\ast.
\end{align*}
For convenience, if $\operatorname{Ad}(U)$ keeps a subspace of $B(L^2(\mathbb{R}))$ invariant, we denote its restriction to the subspace by the same notation. The main theorem of this section is the following. 
\begin{thm}\label{apisom}
Let $\Phi$ be an isometric automorphism  of $A_p$. Then $\Phi$ has the form
\begin{equation}\label{automapal}
\Phi(M_\lambda D_\mu)= c(\mu) d(\lambda) \operatorname{Ad}(V_t) (M_\lambda D_\mu),\, \lambda,\mu\in \mathbb{R}^+
\end{equation}
where $t\in\mathbb{R}$ and $c,d$ are characters of the discrete group of the real numbers. Moreover, the formula \eqref{automapal} gives always a well-defined isometric automorphism of $A_p$.
\end{thm}
Note that in the special case where the characters $c,d$ are continuous in the standard norm of the reals, then their respective automorphisms are unitarily implemented by  $M_\lambda$ and $D_\mu$, for some $\lambda,\mu\in\mathbb{R}$. 
The idea of the proof is to work with the induced homeomorphism of the maximal ideal space of the commutative algebra $A_p/\mathfrak{C}_p$, where $\mathfrak{C}_p$ is the commutator ideal of $A_p$. Similar arguments for the case of crossed products by $\mathbb{Z}^+$ can be found in \cite{pow1, was}. The first step is to identify the commutator ideal $\mathfrak{C}_p$. Define the contractions $E_n, Z_m$ as
in the previous section and the character $x_\infty$ of $AAP(\mathbb{R})$, as it was defined in \ref{charxinfty}. 
\begin{lemma}\label{comideapbaby}
The  commutator ideal 
 $\mathfrak{C}_p$ is equal to the set 
\begin{align*}
\{\alpha\in A_p : E_0(\alpha)=0,Z_0(\alpha)=0\}.
\end{align*}
\end{lemma}
\begin{proof}
If $\alpha=xy-yx\in\mathfrak{C}_p$, then evidently $E_0(\alpha)=Z_0(\alpha)=0$. On the other hand, for every $\lambda, s> 0$  with   $\lambda s$ not equal to  $2n\pi\, (n\in\mathbb{N})$, we have $e^{i\lambda x}=f_s-f_s\circ \tau_s$, where $f_s=e^{i\lambda x}(1-e^{-i\lambda s})^{-1}$. Hence $e^{i\lambda x}D_s\in\mathfrak{C}_p$,
for such $\lambda, s$. Since $\mathfrak{C}_p$ is an ideal it follows that $e^{i\lambda x}D_s\in\mathfrak{C}_p$  for every $\lambda, s >0$. Since these two sets have the same generators (as ideals), the proof is complete.
\end{proof}

\begin{lemma}
$A_p/\mathfrak{C}_p=\{M_f+D_g+\mathfrak{C}_p\,:\,f,g\in AAP(\mathbb{R})\} $.
\end{lemma}
\begin{proof}
It suffices to prove that the RHS is closed. Let $(a_n)_{n\in\mathbb{N}}$ be a sequence, such that $a_n=M_{f_n}+D_{g_n}+\mathfrak{C}_p$ converging to some $a\in A_p/\mathfrak{C}_p$, that is 
\begin{align*}
\inf_{u\in\mathfrak{C}_p}\|a_n-a+u\|\rightarrow 0, \text{ as } n\rightarrow +\infty.
\end{align*}
We may assume that $E_0(D_{g_n})=0$. Since $E_0$ is contractive we have that 
\begin{align*}
&\|E_0(a_n)-E_0(a)\|\leq \|a_n-a+u\|, \forall u\in\mathfrak{C}_p\\
\Rightarrow & \|M_{f_n}-E_0(a)\|\leq \inf_{u\in\mathfrak{C}_p}\|a_n-a+u\|\rightarrow 0, \text{ as } n\rightarrow \infty.
\end{align*}
Similarly, we get that $\|D_{g_n}-[Z_0(a)-Z_0(E_0(a))]\|\rightarrow 0, \text{ as } n\rightarrow \infty$, so $\|a_n-[E_0(a)+Z_0(a)-Z_0(E_0(a))] +\mathfrak{C}_p\|$ converges to $0$, as $n$ goes to infinity. Hence
$a=E_0(a)+Z_0(a)-Z_0(E_0(a))+\mathfrak{C}_p$.
\end{proof} 
Let now $\Phi\in \operatorname{Aut}(A_p)$. Then $\Phi$ induces an automorphism $\tilde{\Phi}\in \operatorname{Aut}(A_p/\mathfrak{C}_p)$ and a homeomorphism $\gamma_0$  between the maximal ideals that contain $\mathfrak{C}_p$, defined by 
\begin{align*}
\gamma_0(\zeta)(\alpha+\mathfrak{C}_p)= \zeta(\tilde{\Phi}(\alpha+\mathfrak{C}_p)), \,\zeta\in\mathfrak{M}(A_p/\mathfrak{C}_p).
\end{align*}

Here, we use the fact that every maximal ideal that contains the commutator ideal is the kernel of a character of the algebra. We want to determine these characters. Write
$AAP_1$ and $AAP_2$ for the function algebras,  both isometrically
isomorphic to $AAP(\mathbb{R})$, that are generated by the multiplication and translation
unitary semigroups, respectively.
Define the mapping
\begin{align*}
\mathfrak{M}(A_p)\rightarrow \mathfrak{M}(AAP_1) \times \mathfrak{M}(AAP_2):\zeta\mapsto 
(\zeta|_{AAP_1}, \zeta|_{AAP_2}),
\end{align*}
where the codomain carries the usual product topology.
\begin{lemma}
This map is a homeomorphism onto the subset
\begin{align*}
(\mathfrak{M}(AAP_1) \times \{x_\infty\})\cup (\{x_\infty\}\times \mathfrak{M}(AAP_2)).
\end{align*}
\end{lemma}
\begin{proof}
Let $\zeta\in\mathfrak{M}(A_p)$, such that $\zeta(M_\lambda)\neq 0$, for some $\lambda>0$.   Then
it follows by the Weyl relations that $\zeta(D_\mu)=0$, for every $\mu>0$. Similarly with the roles of $M_\lambda$ and $D_\mu$ reversed. Hence $\zeta$ maps into the set.
 On the other hand, let $(z,x_\infty)$ be a point in the union set. Define on the generalized trigonometric polynomials the multiplicative linear functional $\zeta$ by
 \begin{align*}
   \zeta \left(\sum_{\lambda,\mu\in F} c_{\lambda,\mu}M_\lambda D_\mu\right)=&
   \sum_{\lambda,\mu\in F} c_{\lambda,\mu}z(M_\lambda) x_\infty(D_\mu)=\\&=
   \sum_{\lambda\in F} c_{\lambda,0} z(M_\lambda).
\end{align*}
But then  $\zeta=z\circ E_0$, so it is bounded and extends to a character of $A_p$.
Similarly, we have that for every point $(x_\infty,z)$ corresponds the character $z\circ Z_0$.   
It remains to show that the map is injective and homeomorphic, but this is routine.   
\end{proof}   

Let $\chi_\infty$ be the preimage of $(x_\infty,x_\infty)$. This the "first coefficient character" on $A_p$ 
\begin{align*}
   \chi_\infty \left(\sum_{\lambda,\mu\in F} c_{\lambda,\mu}M_\lambda D_\mu\right)=c_{0,0}.
\end{align*}
Now,  the maximal ideal space of $AAP(\mathbb{R})$ is the compact topological space $\mathbb{R}_B\times [0,\infty)\cup \{\infty\}$, where $\mathbb{R}_B$ is the Bohr compactification of the real numbers (see Theorem 12.4 in \cite{bks}). Write $\Delta_1, \Delta_2$ for the maximal ideal spaces of $AAP_1$ and $AAP_2$, respectively. 
 Hence, the maximal ideals of $A_p$ that contain $\mathfrak{C_p}$ form the connected topological space 
 \begin{align*}
 \Delta_1 \sqcup _{\chi_\infty} \Delta_2.
 \end{align*}
\begin{lemma}
 $\gamma_0$ fixes $\chi_\infty$. Moreover, either $\gamma_0(\Delta_1)=\Delta_1$, or $\gamma_0(\Delta_1)=\Delta_2$. 
\end{lemma}
\begin{proof}
Given $x\in\mathbb{C}^+\cup\{\infty\}$, let $z_x\in\mathfrak{M}(AAP(\mathbb{R}))$ be the evaluation character at $x$ and $\zeta_x, \eta_x$ be the preimage of the points $(z_x,x_\infty)$ and $(x_\infty,z_x)$, respectively. Note that the set 
\begin{align*}
\mathfrak{M}_{ev}(A_p)=\{ \zeta_x, \eta_x : x\in \mathbb{C}^+\cup\{\infty\}\} 
\end{align*}
is dense in  $\Delta_1 \sqcup _{\chi_\infty} \Delta_2$. Also, with the relative product topology,  this is homeomorphic to the space
\begin{align*}
(\mathbb{C}^+ \times \{\infty\}) \cup (\{\infty\} \times \mathbb{C}^+) \cap \{(\infty, \infty)\}.
\end{align*}

Since $\mathfrak{M}_{ev}(A_p)$ is connected, so is the entire character space $\mathfrak{M}(A_p)$ and its homeomorphic space 
 $\Delta_1 \sqcup _{\chi_\infty} \Delta_2$. If we remove the point $\chi_\infty$, then the character space, with the relative topology, fails to be connected. We claim that $\chi_\infty$ is the only point in the character space with this topological property. 

If $\chi\neq \chi_\infty$ is in $\mathfrak{M}_{ev}(A_p)$, then the set of the remaining evaluation characters, with the relative topology, remains connected, and it contains $\chi$ in its closure. Hence the space 
 $(\Delta_1 \sqcup _{\chi_\infty} \Delta_2)\backslash\{\chi\}$ remains connected. If $\chi$ is a limit character, then once again the space $(\Delta_1 \sqcup _{\chi_\infty} \Delta_2)\backslash\{\chi\}$ contains the dense connected set $\mathfrak{M}_{ev}(A_p)$, so it is connected.

Hence $\chi_\infty$ is a fixed point for homeomorphisms.

Consider now the restriction of the homeomorphism $\gamma_0$ to $(\Delta_1 \sqcup _{\chi_\infty} \Delta_2)\backslash\{\chi_\infty\}$. Since every homeomorphism maps connected components to connected components, the second assertion of the lemma follows.  
\end{proof}
Hence we have two cases.
\begin{description}
\item[Case 1 ] $\gamma_0$ keeps $\Delta_1$ and $\Delta_2$ fixed.
Let $x\in\mathbb{R}$ and let $\zeta_x, \eta_x$ be the characters defined in the proof of the previous lemma.  Since $\gamma_0$ keeps $\Delta_1$ invariant, we have
 \begin{align*} 0=\gamma_0(\zeta_x)(D_\mu)=\zeta_x(\tilde{\Phi}(D_\mu))=E_0(\tilde{\Phi}(D_\mu))(x). 
 \end{align*}
 Hence $E_0(\tilde{\Phi}(D_\mu))=0$ for every $\mu>0$. Therefore $\tilde{\Phi}(D_\mu+\mathfrak{C}_p)=
 D_h+\mathfrak{C}_p$, for some $h\in AAP(R)$. Repeating the argument for $\tilde{\Phi}^{-1}$, we have that $\tilde{\Phi}|_{Z_0(A_p/\mathfrak{C}_p)}$ gives an automorphism of $AAP(\mathbb{R})$. Thus, it follows by Theorem \ref{autapisomisom} that
 \begin{align*}
\tilde{\Phi}(D_\mu+\mathfrak{C}_p)=
 c(\mu) D_{k_1\mu}+\mathfrak{C}_p, \,\text{ for some }k_1>0,\, c(\mu)\in\mathbb{T}.
 \end{align*}
 Applying the same argument on the elements $\tilde{\Phi}(M_\lambda+\mathfrak{C}_p)$, using the $\eta_x$ characters this time, we get 
  \begin{align*}
\tilde{\Phi}(M_\lambda+\mathfrak{C}_p)=
 d(\lambda) M_{k_2\lambda}+\mathfrak{C}_p, \,\text{ for some }k_2>0,\, d(\lambda)\in\mathbb{T}.
 \end{align*}
Hence $\Phi(M_\lambda)=d(\lambda) M_{k_2\lambda}+A$, where $A$ lies in $\mathfrak{C}_p$. 
The following lemma is the only point of the proof of Theorem \ref{apisom} that we will make use of the fact that $\Phi$ is isometric.\\
\begin{lemma}\label{uni}
$\Phi(M_\lambda)=d_\lambda M_{k_2\lambda}$.
\end{lemma}
\begin{proof}
First note that 
\begin{align*}
\|\Phi(M_\lambda)\|=\|M_\lambda\|=1=\|d(\lambda) M_{k_2\lambda}\|.
\end{align*}
If suffices to prove that every Fourier coefficient of $A$ is zero. We consider the left regular representation $(\tilde{id},\Lambda)$ of the crossed product. Let $F$ be a finite subset of positive real numbers and $\xi$ a norm one function in $L^2(\mathbb{R})$.
By Proposition \ref{estimforredcros} we have

\begin{align*}
1=\|\Phi(M_\lambda)\|^2&
\geq \sum_{s\in F\cup \{0\}} \|\tau_{-s}(E_s(\Phi(M_\lambda)))\cdot \xi\|^2_{L^2(\mathbb{R})}=\\
&=
\|d(\lambda) e^{ik_2 \lambda x}\cdot\xi\|^2_{L^2(\mathbb{R})}
+\sum_{s\in F} \|\tau_{-s}(E_s(A))\cdot \xi\|^2_{L^2(\mathbb{R})}
=\\&=
1
+\sum_{s\in F} \|\tau_{-s}(E_s(A))\cdot \xi\|^2_{L^2(\mathbb{R})}.
\end{align*}
 
 So $\tau_{-s}(E_s(A))=0$, which implies that $E_s(A)$ is the zero function, for every $s\in F$. Since $F$ was arbitrarily chosen, we obtain by Corollary \ref{uniqfoco} that $A=0$.
\end{proof}

Similarly using the left regular representation that corresponds to the $(\rho\rtimes M)$ representation of the crossed product, we obtain $\Phi(D_\mu)=c(\mu) D_{k_1\mu}$.

Now the Weyl relations yield 
\begin{align*}
\Phi(M_\lambda D_\mu)=\Phi(e^{i\lambda \mu} D_\mu M_\lambda).
\end{align*}
The LHS gives 
\begin{align*}
\Phi(M_\lambda D_\mu)&=\Phi(M_\lambda)\phi(D_\mu)=d(\lambda) M_{k_2\lambda}c(\mu) D_{k_1\mu}=\\&= d(\lambda) c(\mu) e^{i\lambda k_1k_2 \mu} D_{k_1 \mu} M_{k_2\lambda}, 
 \end{align*}
 while the RHS is
 \begin{align*}
 \Phi(e^{i\lambda \mu} D_s M_\lambda)=
 e^{i\lambda \mu}c(\mu) D_{k_1 \mu} d(\lambda) M_{k_2\lambda}=
 e^{i\lambda \mu} d(\lambda) c(\mu) D_{k_1 \mu} M_{k_2\lambda}.
 \end{align*}
 Therefore $k_1 k_2=1$ and so the automorphisms $\Phi(M_\lambda D_\mu)=M_{k_2\lambda}D_{k_1\mu}$ correspond to the automorphisms $\operatorname{Ad}(V_t)$, by taking $t=\log k_1$. Each automorphism of the above form is induced by a covariant respesentation of $(AP(\mathbb{R}),\mathbb{R}_d,\tau)$, so by the universal property of the crossed product it extends to an algebra automorphism of $B_p$.

Define on $L^2(\mathbb{R})$ the covariant representation $(y_{d,t},w_{c,t})$ of the $C^\ast-$dynamical system $(AP(\mathbb{R}),\mathbb{R}_d, \tau)$, where
\begin{align*}
y_{d,t}:AP(\mathbb{R})\rightarrow B(L^2(\mathbb{R})): f\mapsto M_{\tilde{\operatorname{id}}\rtimes u_{d,e^t}(f)},
\end{align*}
where $\tilde{\operatorname{id}}\rtimes u_{d,e^t}$ are given in equation \eqref{eqautapbab}, and
\begin{align*}
w_{c,t}:\mathbb{R}_d\rightarrow B(L^2(\mathbb{R})): \mu\mapsto c(\mu) D_{\mu e^{-t}}. 
\end{align*}
Indeed, the pair $(y_{d,t},w_{c,t})$ is a covariant representation, since

\begin{align*}
w_{c,t}(\mu) y_{d,t}(e^{i\lambda x})w_{c,t}(-\mu)&= c(\mu) D_{\mu e^{-t}}d(\lambda) M_{\lambda e^t}c(-\mu)D_{-\mu e^{-t}}=\\&
= e^{-i\lambda \mu} d(\lambda) M_{\lambda e^{t}}=e^{-i\lambda \mu}y_{d,t}(e^{i\lambda x}).
= y_{d,t}(\tau_\mu(e^{i\lambda x}))
\end{align*}
Hence, by the universal property of the crossed product, we obtain the induced isometric automorphism $y_{d,t}\rtimes w_{c,t}$ of $B_p$ that satisfies 
\begin{align*}
 M_\lambda D_\mu\mapsto d(\lambda) c(\mu) M_{\lambda e^t} D_{\mu e^{-t}}.
\end{align*}
 It is evident now that the automorphism $\Phi$ given in relation \eqref{automapal} is of the form $y_{d,t}\rtimes w_{c,t}$ (restricted to $A_p$), for some $t\in\mathbb{R}$ and $c,d\in\mathbb{R}_B$. 

\item[Case 2 ] $\gamma_0$ flips $\Delta_1$ and $\Delta_2$.
 Repeating the argument of the previous case, we end up with  
 \begin{align*}
 \Phi(M_\lambda)= d(\lambda) D_{k_1\lambda} ~~\text{ and }~~
 \Phi(D_\mu)= c(\mu) M_{k_2\mu}.
 \end{align*}
 Applying again the Weyl commutation relations, we calculate 
\begin{align*}
 \Phi(M_\lambda)\Phi(D_\mu)= e^{i\lambda \mu} \Phi(D_\mu)\Phi(M_\lambda)&\Leftrightarrow d(\lambda) c(\mu)D_{k_1 \lambda} M_{k_2 \mu}= e^{i\lambda \mu}d(\lambda) c(\mu)M_{k_2 \mu} D_{k_1 \lambda}\\
 &\Leftrightarrow 
 d(\lambda) c(\mu)D_{k_1 \lambda} M_{k_2 \mu}= e^{i\lambda \mu(1+ k_1 k_2)}d(\lambda) c(\mu)D_{k_1 \lambda} M_{k_2 \mu}
\end{align*}
which implies that $k_1 k_2 =-1$, but this is impossible, since $k_1, k_2$ are both positive real numbers.
 \end{description}
This completes the proof of Theorem \ref{apisom}. \qed

\section{Triple semigroup algebras}

As described in the previous section, the dilation operators $\{V_t\,:\, t\in\mathbb{R}\}$ implement isometric automorphisms of the $C^\ast$-algebra $B_p$. 
 Let $G$  be the discrete group $\mathbb{R}_d$ or $\mathbb{Z}$ and $(B_p, G,v)$ be the $C^\ast$-dynamical system, where $v$ is the group of automorphisms that are unitarily implemented by the operators $V_t$
\begin{align*}
v:G\rightarrow \operatorname{Aut}(B_p):t\mapsto v_t= \operatorname{Ad}(V_t).
\end{align*}
Hence, this enables us to define the crossed product, denoted by $B_p\times_v G$. Denote by $H_k$  the contraction from $B_p\times_v G$ onto $B_p$
\begin{align*}
H_k(\sum_{\substack{\lambda,\mu, t\in F\\F\text{ finite}}}(\epsilon_t\otimes c_{\lambda,\mu, t} M_\lambda D_\mu))= 
\sum_{\substack{\lambda,\mu, t\in F\\F\text{ finite}}}c_{\lambda,\mu, k} M_\lambda D_\mu.
\end{align*}
Our next goal is to show that the norm closed algebra
\begin{align*}
B_{ph}^G:=\|\cdot\|\textrm{-alg}\{M_\lambda, D_\mu, V_t\,:\, \lambda,\mu\in \mathbb{R}, t\in G\}
\end{align*}
is  isometrically isomorphic to $B_p\times_v G$. 
By the universal property of the crossed product we have the representation 
\begin{align*}
((\pi\rtimes D)\rtimes V )\sum_{\substack{\lambda,\mu, t\in F\\F\text{ finite}}}(\epsilon_t\otimes c_{\lambda,\mu, t} M_\lambda D_\mu) \mapsto \sum_{\substack{\lambda,\mu, t\in F\\F\text{ finite}}} c_{\lambda,\mu, t} M_\lambda D_\mu V_t.
\end{align*}
The following proposition is the key to prove that the above representation is actually an isometric isomorphism.

\begin{prop}\label{expect}
Given $t_0\in G$, the mapping 
  \begin{align*}
  \sum_{\substack{\lambda,\mu, t\in F\\F\text{ finite}}} c_{\lambda,\mu, t} M_\lambda D_\mu V_t\mapsto 
   \sum_{\substack{\lambda,\mu, t\in F\\F\text{ finite}}} c_{\lambda,\mu,t_0} M_\lambda D_\mu 
   \end{align*}
is contractive, so it extends to a linear contraction $\tilde{H}_{t_0}$ on $B_{ph}^G$.
  \end{prop}
  \begin{proof}
  It suffices to prove it for $t_0=0$.
 By Poincare's recurrence theorem \cite{fur}, there exists an increasing unbounded sequence $\{M_n\}_{n\in\mathbb{N}}$ of natural numbers, such that 
  \begin{align*}
  e^{i\lambda M_n}\rightarrow 1, \text{ as } n\rightarrow\infty \text{ and for all }  \lambda\in F.
\end{align*}
Since $D_{M_n}V_t D_{M_n}^\ast\stackrel{\textrm{WOT}}{\rightarrow} 0$ for every $t\neq 0$, one can check that
\begin{align*}
\lim_{n\rightarrow \infty}\langle D_{M_n} c_{\lambda,\mu, t} M_\lambda D_\mu V_t D_{M_n}^\ast f,g\rangle 
=
\left\{
	\begin{array}{ll}
		\langle c_{\lambda,\mu,0} M_\lambda D_\mu f,g\rangle,  & \mbox{if } t=0 \\
		0 & \mbox{if } t \neq 0
	\end{array}
\right. 
\end{align*}
Hence
\begin{align*}
\sum_{\substack{\lambda,\mu, t\in F\\F\text{ finite}}} c_{\lambda,\mu, t} D_{M_n}M_\lambda D_\mu V_t D_{M_n}^\ast\stackrel{\textrm{WOT}}{\rightarrow} 
 \sum_{\substack{\lambda,\mu, t\in F\\F\text{ finite}}} c_{\lambda,\mu,0} M_\lambda D_\mu.
\end{align*}
Therefore, the proof follows by observing that 
\begin{align*} 
\bigg\langle\sum_{\substack{\lambda,\mu, t\in F\\F\text{ finite}}} c_{\lambda,\mu, t} D_{M_n}M_\lambda D_\mu V_t D_{M_n}^\ast
f,g\bigg\rangle \leq \|\sum_{\substack{\lambda,\mu, t\in F\\F\text{ finite}}} c_{\lambda,\mu, t} M_\lambda D_\mu V_t \|\,\|f\|\,\|g\|,
\end{align*}
for all $f,g \in L^2(\mathbb{R})$.
\end{proof}

\begin{cor}
The map $(\pi\rtimes D)\rtimes V$ is an isometric $\ast$-isomorphism of the C$^\ast$-algebra $B_p\rtimes_v \mathbb{Z}$.
\end{cor}
\begin{proof}
Let $X\in B_p\rtimes_v \mathbb{Z}$, such that $((\pi\rtimes D)\rtimes V)(X)=0$. Then by the previous proposition we get
$\tilde{H}_k (((\pi\rtimes D)\rtimes V)(X))=0$, for every $k\in G$.
But $\tilde{H}_k \circ((\pi\rtimes D)\rtimes V)=(\pi\rtimes D)\circ H_k$, since the equality holds for trigonometrical polynomials.
Therefore $((\pi\rtimes D)\circ H_k)(X)=0$, which implies that $H_k(X)=0$, so $X=0$.
Hence the representation $(\pi\rtimes D)\rtimes V$ is faithful, so isometric.
\end{proof}
We denote by $A_{ph}^{G^+}$ the norm closed algebra that is generated by the semigroups of $M_\lambda, D_\mu, V_t$, where $\lambda,\mu\in \mathbb{R}^+, t \in G^+$. The algebra $A_{ph}^{\mathbb{Z}^+}$ is called the \textbf{partially discrete triple semigroup algebra}, while the algebra $A_{ph}^{\mathbb{R}^+}$  is called the \textbf{triple semigroup algebra}.

Let $\mathfrak{C}_{ph}^{G^+}$ be the commutator ideal of $A_{ph}^{G^+}$. To describe $\mathfrak{C}_{ph}^{G^+}$ we need first the following lemma. 

Fix $t>0$ and let $J_t$ be the closed ideal of $AAP(\mathbb{R})$ generated by the functions of the form 
\begin{align*}
e^{i\lambda x} - \phi_{0,e^t}(e^{i\lambda x})=e^{i\lambda x}-e^{i\lambda e^t x}, 
\end{align*}
for $\lambda>0$.
\begin{lemma}
The ideal $J_t$ is equal to the ideal $I_0=\{f\in AAP(\mathbb{R}) : f(0)= x_\infty(f)=0\}$.
\end{lemma}
\begin{proof}
It is clear that $I_0$ contains $J_t$. To prove the inverse inclusion, note that $I_0$ has codimension 2, and so it suffices to show that
the same holds for $J_t$. Define the subspace $\tilde{J_t}=span\{a+ c e^{ix}: a\in J_t, c\in \mathbb{C}\}$. We claim that $\tilde{J_t}$ is closed.

Let $\{a_n+c_n e^{ix}\}_n$ be a convergent sequence, such that $a_n\in J_t$ and $c_n\in \mathbb{C}$.
We claim that the limit of the sequence, say $a$, lies in $\tilde{J_t}$. Denote by $x_1$ the character of $AAP(\mathbb{R})$ given by the formula 
\begin{equation}\label{charx1ev0}
x_1(f)\mapsto f(0).
\end{equation}
 Hence $a_n(0)+c_n \rightarrow a(0)$. However, since $a_n\in J_t$, it follows that  $a_n(0)=0$, for all $n\in \mathbb{N}$. Therefore 
\begin{align*}
c_n \rightarrow a(0) \,\,\,\, \Rightarrow c_n e^{ix} \rightarrow a(0) e^{ix},\,\,\,\text{ as } n\rightarrow \infty.
\end{align*}
So $a_n=a_n + c_n e^{ix}- c_n e^{ix}\rightarrow a- a(0) e^{ix}$. Since $J_t$ is closed, it contains $a- a(0) e^{ix}$. Hence
\begin{align*}
a= a- a(0) e^{ix} + a(0) e^{ix} \,\, \in \tilde{J_t},
\end{align*}
so $\tilde{J_t}$ is closed.

Hence, it suffices to prove that
\begin{align*} 
\tilde{J_t}=\overline{span\{e^{i\lambda x}: \lambda>0\}}^{\|\cdot\|_\infty}.
\end{align*}
 Since $J_t$ is an ideal in $AAP(\mathbb{R})$, we get
\begin{align*}
e^{i(\kappa+\lambda)x}- e^{i(\kappa+\lambda e^t)x}-e^{ix} \in \tilde{J_t},\text{ for all } \kappa,\lambda\in(0,\infty).
\end{align*}
Choose $\kappa+\lambda=1$, so $e^{i\rho x}\in \tilde{J_t}$  for every $\rho \in [1,e^t)$. Thus, by induction, we have that 
\begin{align*}
e^{i\rho e^{(n-1)t}x}-e^{i\rho e^{nt}x}-e^{i\rho e^{(n-1)t}x}=-e^{i\rho e^{nt}x}\in \tilde{J_t},
\end{align*}
and 
\begin{align*}
e^{i\rho e^{-nt}x}-e^{i\rho e^{-(n-1)t}x}+e^{i\rho e^{-(n-1)t}x}= e^{i\rho e^{-nt}x}\in \tilde{J_t},
\end{align*}
 for every $\rho\in[1,e^t)$ and $n\in \mathbb{N}$. 
Hence $e^{i\lambda x}\in\tilde{J_t}$, for all $\lambda\in(0,\infty)$, and hence the proof is complete.
\end{proof}

\begin{prop}
The commutator ideal $\mathfrak{C}_{ph}^{G^+}$ is equal to the set 
\begin{equation}\label{cphdescrip}
\operatorname{ker}(E_0\circ H_0)\cap \operatorname{ker}(Z_0\circ H_0) \cap \bigcap_{t\in G^+}\left(
\operatorname{ker}(\chi_\infty\circ H_t) \cap \operatorname{ker}(x_1\circ E_0\circ H_t)\cap \operatorname{ker}(x_1\circ Z_0\circ H_t)\right). 
\end{equation}
\end{prop}
\begin{proof}
Let $I$ be the set described in \eqref{cphdescrip}. Since $I$ is the intersection of kernels of bounded linear operators, it is closed.
One can check that if $X=  \sum\limits_{\substack{\lambda,\mu, t\in F\\F\text{ finite}}} c_{\lambda,\mu, t} M_\lambda D_\mu V_t$ is a trigonometric polynomial in $I$, then it satisfies
\begin{enumerate}
	\item 
	$c_{\lambda,0,0}=c_{0,\mu,0}=c_{0,0,t}=0, \text{ for all } \lambda,\mu\in\mathbb{R}^+, t\in G^+$;
	\item 
	$\sum\limits_{\lambda}c_{\lambda,0,t}=\sum\limits_{\mu}c_{0,\mu,t}=0, \text{ for all } t\in G^+$.
	\end{enumerate}
 It is elementary to show that if $X,Y$ trigonometric polynomials in $A_{ph}^{G^+}$, then  $XY-YX\in I$. Since multiplication is jointly continuous with respect to the operator norm, it follows by the density of trigonometric polynomials in $A_{ph}^{G^+}$ that $XY-YX\in I$, for every $X,Y\in A_{ph}^{G^+}$. Similarly, working first with trigonometric polynomials, we obtain that $I$ is closed under the ideal operations.

For the converse inclusion, let $X\in I$. By Theorem \ref{Cesaro covergence}, it suffices to show that $H_t(X)V_t\in \mathfrak{C}_{ph}^{G^+}$, for every $t\in G^+$. As we proved in Lemma \ref{comideapbaby}, one can check that 
\begin{align*}
H_t(X) - E_0(H_t(X))-Z_0(H_t(X)) + (E_0\circ Z_0)(H_t(X))\in \mathfrak{C}_{ph}^{G^+},
\end{align*}
since it lies in $\mathfrak{C}_{p}$. Moreover, we obtain by the definition of $I$ that $(E_0\circ Z_0)(H_t(X))=(\chi_\infty \circ H_t)(X)=0$, so it follows that 
\begin{align*}
H_t(X)V_t - E_0(H_t(X))V_t-Z_0(H_t(X))V_t\in \mathfrak{C}_{ph}^{G^+}.
\end{align*}
Hence it suffices to show that $E_0(H_t(X))V_t$ and $Z_0(H_t(X))V_t$ lie in $\mathfrak{C}_{ph}^{G^+}$.

Write $ E_0(H_t(X))V_t= M_f V_t$, for some $f\in AAP(\mathbb{R})$. Since $X\in I$, $f$ satisfies the properties $x_\infty(f)=f(0)=0$. So by the previous lemma there exist $g_n\in AAP(R),\, n\in\mathbb{N}$,  such that $f=\lim\limits_{n}(g_n - \phi_{0,e^t}(g_n))$, which implies that $M_f V_t=\lim\limits_n(M_{g_n} V_t -V_t M_{g_n})$, so $M_f V_t$ lies in $\mathfrak{C}_{ph}^{G^+}$. Similarly every element $D_f V_t\in I$ belongs to $\mathfrak{C}_{ph}^{G^+}$, so our proof is complete.
\end{proof}


Before this subsection ends, we prove the existence of two more contractive maps, which will be helpful in the next section.
\begin{prop}\label{expect2}
The maps
\begin{align*}
  \sum_{\substack{\lambda,\mu, t\in F\\F\text{ finite}}} c_{\lambda,\mu, t} M_\lambda D_\mu V_t\mapsto 
   \sum_{\substack{\lambda,t\in F\\F\text{ finite}}} c_{\lambda,0,t} V_t\\
   \sum_{\substack{\lambda,\mu, t\in F\\F\text{ finite}}} c_{\lambda,\mu, t} M_\lambda D_\mu V_t\mapsto 
   \sum_{\substack{\mu, t\in F\\F\text{ finite}}} c_{0,\mu,t} V_t
   \end{align*}
 are contractive.
\end{prop}
\begin{proof}
The proof uses similar arguments as in  Proposition \ref{expect}, working now with the WOT-limits
\begin{align*}
\sum_{\substack{\lambda,\mu, t\in F\\F\text{ finite}}} c_{\lambda,\mu, t} V_n^\ast M_\lambda D_\mu V_t V_n\stackrel{\textrm{WOT}}{\rightarrow} 
 \sum_{\substack{\lambda, t\in F\\F\text{ finite}}} c_{\lambda,0,t} V_t \\
\sum_{\substack{\lambda,\mu, t\in F\\F\text{ finite}}} c_{\lambda,\mu, t} V_n M_\lambda D_\mu V_t V_n^\ast\stackrel{\textrm{WOT}}{\rightarrow} 
 \sum_{\substack{\mu, t\in F\\F\text{ finite}}} c_{0,\mu,t} V_t,
\end{align*}
as $n\rightarrow \infty$.
\end{proof}

\subsection{The algebra $A_{ph}^{\mathbb{Z}^+}$}
We focus now on the partly discrete triple semigroup algebra $A_{ph}^{\mathbb{Z}^+}$. In order to determine the isometric automorphisms of $A_{ph}^{\mathbb{Z}^+}$, we work again on the induced homeomorphism of the character space onto itself. Define the characters $x_1\in\mathfrak{M}(AAP(\mathbb{R}))$, such that $x_1(f)=f(0)$, and $\chi_\infty=(x_\infty,x_\infty)$ as before. Let also 
$y_0$ be the character in the disc algebra   $A(\mathbb{D})$ (see \cite{hof}), given by  $y_0(f)=f(0)$. 
\begin{prop}
The mapping 
\begin{align*}
\psi:\mathfrak{M}(A_{ph}^{\mathbb{Z}^+})\rightarrow \mathfrak{M}(AAP_1)\times \mathfrak{M}(AAP_2)\times\mathfrak{M}(A(\mathbb{D})):\chi\mapsto (\chi\big|_{AAP_1},\chi\big|_{AAP_2},\chi\big|_{A(\mathbb{D})})
\end{align*}
is continuous into the subset
\begin{align*}
(\mathfrak{M}(AAP_1) \times \{x_\infty\}\times \{y_0\})&\cup (\{x_\infty\}\times \mathfrak{M}(AAP_2)\times \{y_0\})\cup\\ 
\cup(\{x_1\}\times\{x_\infty\}\times \mathfrak{M}(A(\mathbb{D})))
\cup (\{x_\infty\}\times\{x_1\}\times &\mathfrak{M}(A(\mathbb{D})))\cup (\{x_\infty\}\times\{x_\infty\}\times \mathfrak{M}(A(\mathbb{D}))).
\end{align*}
\end{prop}
\begin{proof}
Let  $\chi$ be a character in $\mathfrak{M}(A_{ph}^{\mathbb{Z}^+})$. Then
\begin{align*} 
\chi\big|_{A_p}\in\mathfrak{M}(A_p)\,\,\text{ and }\,\,
\chi\big|_{\|\cdot\|\textrm{-alg}\{V_t:t\in\mathbb{Z}^+\}}\in\mathfrak{M}(A(\mathbb{D})).
\end{align*}
One can check that if $\chi\big|_{A_p}$ does not correspond to a point $\{\chi_\infty,(x_1,x_\infty),(x_\infty,x_1)\}$, then by the commutation relations we get that $\chi(V_t)=0$, for all positive $t$.  

On the other hand, if $\chi\big|_{\|\cdot\|\textrm{-alg}\{V_t:t\in\mathbb{Z}^+\}}\neq 0$, then by the commutation relations we have three cases for $\chi\big|_{A_p}\in\mathfrak{M}(A_p)$:
\begin{enumerate}
	\item $\chi(M_\lambda)=1$ and $\chi(D_\mu)=0$, which corresponds to the character $(x_1,x_\infty)$ in  $\mathfrak{M}(A_p)$.
	\item $\chi(M_\lambda)=0$ and $\chi(D_\mu)=1$, so we get the character $(x_\infty,x_1)$.
	\item $\chi(M_\lambda)=\chi(D_\mu)=0$, which gives $\chi_\infty$.
\end{enumerate}
Hence the mapping $\psi$ is well defined. Continuity is evident, so the proof is complete.
\end{proof}

Note that every element in the codomain of $\psi$ corresponds to a multiplicative linear functional defined on the non-closed algebra of trigonometric generalized polynomials.
Write once again $\Delta_1$, $\Delta_2$ for the sets $\mathfrak{M}(AAP_1) \times \{x_\infty\}\times \{y_0\}$ and  $\{x_\infty\}\times \mathfrak{M}(AAP_2)\times \{y_0\}$ respectively.
If $\chi$ is such a multiplicative functional, then the contraction $H_0$ yield that $\chi$ is bounded and extends to a character of $A_{ph}^{\mathbb{Z}^+}$. Therefore, any maximal ideal of $A_p$ corresponding to a point $(\Delta_1\sqcup_{\chi_\infty}\Delta_2)\backslash\{\chi_\infty,(x_1,x_\infty),(x_\infty,x_1)\}$ is contained in a unique maximal ideal in $A_{ph}^{\mathbb{Z}^+}$.
Similarly, by  Lemma \ref{expect2} any multiplicative functional of the form $(x_1,x_\infty,y), (x_\infty,x_1,y)$, with $y\in \mathfrak{M}(A(\mathbb{D}))$, is bounded. 
 We denote by $\Delta_3$ the sets of characters that give $\chi(M_\lambda)=1$, for all $\lambda$, and 
by  $\Delta_4$ the characters that satisfy $\chi(D_\mu)=1$, for all $\mu$.

The pursuit of the continuity of the remaining multiplicative functionals (on the dense subalgebra) that correnspond to the points $(x_\infty,x_\infty,y)$, we write $\Delta_0$,  is more subtle and it remains unclear to the author if this formula can generate a bounded character of $A_{ph}^{\mathbb{Z}^+}$. 
\begin{rem}
It is trivial to show that given an element $u$ of the commutator ideal of a commutative Banach algebra $A$, then $\chi(u)=0$ for every character $\chi$ of $A$. The opposite direction is not true in the case that $A$ contains quasinilpotent elements. A complication with $A_{ph}^{\mathbb{Z}^+}$ is that we cannot determine if the elements of the form $V_t-M_\lambda V_t -D_\mu V_t + \mathfrak{C}_{ph}^{\mathbb{Z}^+}, \lambda,\mu\in \mathbb{R}^+, t\in \mathbb{Z}^+$ are quasinilpotent, a property which turns out to be equivalent to the continuity of specific elements in $\Delta_0$. 
\end{rem}

We now obtain a partial identification of the character space of 
$A_{ph}^{\mathbb{Z}^+}$, which is sufficient for our main results in the next section.
See Figure \ref{topsp}.


\begin{figure}[h!]
\begin{center}
\pgfmathsetmacro{\radius}{1}
\pgfmathsetmacro{\thetavec}{0}
\pgfmathsetmacro{\phivec}{0}


\begin{tikzpicture}[scale=2]
\draw[dashed] (\radius,0,0) arc (-90:90:\radius);
\shade[ball color=blue!10!white,opacity=0.2] (1cm,2cm) arc (90:270:5mm and 1cm) arc (-90:90:1cm and 1cm);
\draw (1,1) ellipse (0.5cm and 1cm);
\draw (2, 1, 0) node [circle, fill=blue, inner sep=.05cm] () {};
\draw (0.5, 1, 0) node [circle, fill=blue, inner sep=.05cm] () {};
\draw (3.5, 1, 0) node [circle, fill=blue, inner sep=.05cm] () {};
\node [below left] at (0.3,1) {$\Delta_3$};
\node [below left] at (1.2,1) {$\Delta_1$};
\node [below left] at (2.1,1) {$\Delta_0$};
\node [below left] at (3,1) {$\Delta_2$};
\node [below left] at (4.1,1) {$\Delta_4$};
\draw[black, thick, dashed,fill=yellow, fill opacity=0.2] (0.5,1) ellipse (0.3cm and 1cm);
\draw[black, thick, dashed,fill=yellow, fill opacity=0.2] (2,1) ellipse (0.3cm and 1cm);
\draw[black, thick, dashed,fill=yellow, fill opacity=0.2] (3.5,1) ellipse (0.3cm and 1cm);
\draw[dashed] (3,2) arc (90:270:\radius);
\shade[ball color=blue!10!white,opacity=0.2] (3cm,0cm) arc (270:90:1cm and 1cm) arc (90:-90:5mm and 1cm);
\draw[dashed] (3,0) arc (270:90:0.5cm and 1cm);
\draw (3,0) arc (-90:90:0.5cm and 1cm);
\end{tikzpicture}
\end{center}
\caption{The topological space $\Delta_0 \sqcup \Delta_1\sqcup \Delta_2 \sqcup \Delta_3 \sqcup \Delta_4$.}
\label{topsp}
\end{figure}
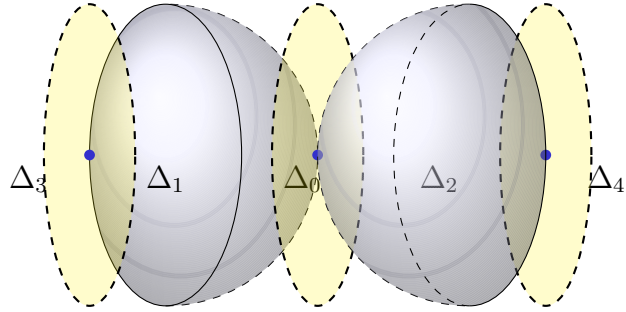


\begin{prop}\label{homot}
The character space $\mathfrak{M}(A_{ph}^{\mathbb{Z}^+})$ has the form  $\tilde{\Delta}_0 \sqcup \Delta_1\sqcup \Delta_2 \sqcup \Delta_3 \sqcup \Delta_4$, where $\tilde{\Delta}_0$ is either the point $\{x_\infty,x_\infty, y_0\}$ or a closed disc in $\Delta_0$.
\end{prop}
\begin{proof}
If there is no continuous character of $\mathfrak{M}(A_{ph}^{\mathbb{Z}^+})$ in $\Delta_0$, apart from $\{x_\infty,x_\infty, y_0\}$, then there is nothing to prove.
Assume now that $\chi$ is a continuous character in $\Delta_0$, so $\chi(V_t)=z^t$ for some $z\neq 0$ in the unit disk. 
Hence 
\begin{align*}
|\sum_{t} \chi(a_t) z^t| \leq \|\sum_{t} a_t V_t\|\, ,\, a_t\in A_p.
\end{align*}
Applying the dual automorphisms $\phi_{e^{i\theta}}$ of $A_p\rtimes_v \mathbb{Z}^+$ for any $\theta\in(0,2\pi)$, it follows that 
\begin{align*}
|\sum_{t\in\mathbb{N}} \chi(a_t) (ze^{i\theta})^t| \leq \|\sum_{t\in\mathbb{N}} e^{i\theta t}a_t V_t\|=\|\sum_{t\in\mathbb{N}} a_t V_t\|.
\end{align*}
Therefore, by the maximum principle, each multiplicative linear functional of the form $a_t V_t\mapsto \chi(a_t)w^t$, where $|w|\leq |r|$, is continuous.
\end{proof}

\begin{thm}\label{isomVt}
The isometric isomorphisms of $A_{ph}^{\mathbb{Z}^+}$ are of the form 
\begin{equation}\label{isomisomform2}
\Phi(M_\lambda)=M_{k_1\lambda},\, \Phi(D_\mu)=D_{k_2\mu} \text{ and } \Phi(V_t)= c(t) V_t,
\end{equation}
where  $k_1 k_2 =1$ and $c:t\mapsto c(t)$ is multiplicative. 
\end{thm}
\begin{proof} 
Let $\Phi$ be an isometric isomorphism of $A_{ph}^{\mathbb{Z}^+}$. Once again we consider the induced homeomorphism 
\begin{align*}
\gamma :\mathfrak{M}(A_{ph}^{\mathbb{Z}^+}) \rightarrow \mathfrak{M}(A_{ph}^{\mathbb{Z}^+}) : \chi\mapsto \chi\circ \Phi^{-1}.
\end{align*}
 It follows by Proposition \ref{homot} that $\gamma$ fixes the subset of characters $\Delta_p=\Delta_1\sqcup_{(x_\infty,x_\infty,y_0)}\Delta_2$.
Hence the ideal $\mathcal{I}=\cap_{\chi\in\Delta_p} \textrm{ker}\chi$ is fixed by $\Phi$. By Proposition \ref{expect} it follows that the quotient algebra $A_{ph}^{\mathbb{Z}^+}/\mathcal{I}$ is isomorphic to $A_p/\mathfrak{C}_p$.
So the naturally induced automorphism $\tilde{\Phi}$ of the quotient algebra satisfies
\begin{align*}
\tilde{\Phi}(M_\lambda + \mathcal{I})&= d(\lambda) M_{k_1 \lambda}+ \mathcal{I} \\
\tilde{\Phi}(D_\mu + \mathcal{I})&= c(\mu) D_{k_2\mu}+ \mathcal{I}
\end{align*}
where $k_1k_2=1$ and $c,d$ are characters of the discrete group of the real numbers. Applying the same argument as in Lemma \ref{uni} we get that $\Phi(M_\lambda)=d(\lambda) M_{k_1 \lambda}$ and $\Phi(D_\mu)=c(\mu) D_{k_2 \mu}$.
Now, since the characters in $\Delta_3$ are continuous, by the commutation relations we get that 
\begin{align*}
\Phi(V_t)\Phi(M_\lambda)=\Phi(M_{\lambda e^t})\Phi(V_t) \Rightarrow d(\lambda)= d(\lambda e^t)\Rightarrow  d(\lambda)=1.
\end{align*}
Similarly, using the continuity of the characters in $\Delta_4$, we get that $c(\mu)=1$, for every $\mu>0$.
The argument to determine the image of the dilation operators is developed entirely on $L^2(\mathbb{R})$.
Since $\Phi( V_t) M_{k_1\lambda}=M_{k_1\lambda e^t}\Phi(V_t)$, if we right multiply both sides by $V_t^\ast$, we get
\begin{align*}
\Phi( V_t) V_t^\ast M_{k_1\lambda e^t}=M_{k_1\lambda e^t}\Phi(V_t)V_t^\ast.
\end{align*}
Hence $\Phi(V_t)V_t^\ast$ commutes with every $M_\lambda,\lambda\in\mathbb{R}$, so it lies in the multiplication algebra $\mathcal{M}_m$, since this algebra is maximal abelian.
Mimicking the same argument for the commutation relation with the translation operator, we get that $\Phi(V_t)V_t^\ast$ is also in the translation algebra $\mathcal{D}_m$.
But the intersection of these two algebras is the multiples of the identity operator, so $\Phi(V_t)=c(t) V_t$.  
We proved that $\Phi$ satisfies $\Phi(M_\lambda)=M_{k_1\lambda},\, \Phi(D_\mu)=D_{k_2\mu}, \Phi(V_t)= c(t) V_t$, where $k_1 k_2 =1$. Moreover since $c$ is multiplicative, we obtain  that $c(t)=e^{i\theta t}$, for some $\theta\in[0,2\pi)$ independent of $t$.
By the universal property of the crossed product, any such mapping can extend to an isometric isomorphism of $A_p\rtimes_v \mathbb{Z}^+$. 
\end{proof}

\begin{thm}
The algebra $A_{ph}^{\mathbb{Z}^+}$ is chiral.
\end{thm}
\begin{proof}
It suffices to prove that $A_{ph}^{\mathbb{Z}^+}$ is not isometrically isomorphic to its conjugate algebra $(A_{ph}^{\mathbb{Z}^+})^\ast$.
If $\Phi$ was such an isomorphism, then following the same proof as in the previous theorem we get that
$\Phi (M_\lambda)=M_{-k_1\lambda}$ and $\Phi(D_\mu)=D_{-k_2\mu}$. But then again, we can prove that $\Phi(V_t) V_t^\ast=c(t) I$, so $\Phi(V_t)=c(t)V_t\notin (A_{ph}^{\mathbb{Z}^+})^\ast$.
\end{proof}

\subsection{The algebra $A_{ph}^{\mathbb{R}^+}$}
 The approach to the triple semigroup algebra is similar to the case of $A_{ph}^{\mathbb{Z}^+}$. Note that the algebra generated by the unitary semigroup $\{V_t\}_{t\geq0}$ is isometrically isomorphic to $AAP(\mathbb{R})$. Writing $AAP_3$ for this algebra, we obtain that the mapping 
\begin{align*}
\mathfrak{M}(A_{ph}^{\mathbb{R}^+})\rightarrow \mathfrak{M}(AAP_1)\times \mathfrak{M}(AAP_2)\times\mathfrak{M}(AAP_3):\chi\mapsto (\chi\big|_{AAP_1},\chi\big|_{AAP_2},\chi\big|_{AAP_3})
\end{align*}
is continuous into the subset
\begin{align*}
(\mathfrak{M}(AAP_1) \times \{x_\infty\}\times \{x_\infty\})&\cup (\{x_\infty\}\times \mathfrak{M}(AAP_2)\times \{x_\infty\})\cup\\ 
\cup(\{x_1\}\times\{x_\infty\}\times \mathfrak{M}(AAP_3))
\cup (\{x_\infty\}\times\{x_1\}\times &\mathfrak{M}(AAP_3))\cup (\{x_\infty\}\times\{x_\infty\}\times \mathfrak{M}(AAP_3)).
\end{align*}
We also keep the notation for $\Delta_1, \Delta_2, \Delta_3, \Delta_4, \Delta_0$ as in the previous section. Note now that each disk is homeomorphic to the topological space $\mathbb{R}_B\times [0,\infty)\cup{\infty}$. Again, the continuity of the characters in $\Delta_1, \Delta_2, \Delta_3$ and $\Delta_4$ follows from Propositions \ref{expect} and \ref{expect2}, while it is unknown to the author if the multiplicative linear functionals in $\Delta_0$ are continuous. Moreover, it remains also unclear if Proposition \ref{homot} holds in this case, since we may have continuous limit characters in $\Delta_0$. 
Nonetheless, let $\chi_{z}$ be the multiplicative functional in $\Delta_0$, that evaluates a function in $AAP3$ to the point $z$ of the upper half plane of $\mathbb{C}$. If $\chi_z$ was continuous, then mimicking the proof of \ref{homot}, we would get that any multiplicative functional of the form $\chi_{w}$, where $Im(w)\geq Im(z)$, is continuous. Moreover, any limit character in the closure of the set $\{\chi_w: Im(w)\geq Im(z)\}$ would be continuous. 

Given now any isometric isomorphism $\Phi$ of $A_{ph}^{\mathbb{R}^+}$, define the induced homeomorphism, say $\gamma$, of the character space $\mathfrak{M}(A_{ph}^{\mathbb{R}^+})$ onto itself. Since the set of limit characters has empty interior, it follows that $\gamma$ permutes the discs. Hence it fixes the set $\Delta_p$ of characters that map the family of the dilation operators $\{V_t\}_{t>0}$ to zero. 
This is the closure of the set of characters of the norm closed parabolic algebra that are extended uniquely in the triple semigroup algebra. 
Hence using the same arguments we get that the restriction of $\Phi$ in $A_p$ is a isometric automorphism of the parabolic algebra. Then, repeating the last argument of the proof of Theorem \ref{isomVt}, we have the corresponding result;
\begin{thm}
The isometric isomorphisms of $A_{ph}^{\mathbb{R}^+}$ are of the form 
\begin{align*}
\Phi(M_\lambda)=M_{k_1\lambda},\, \Phi(D_\mu)=D_{k_2\mu} \text{ and } \Phi(V_t)= c(t) V_t,
\end{align*}
where  $k_1 k_2 =1$ and $c:t\mapsto c(t)$ is multiplicative. Furthermore, the algebra $A_{ph}^{\mathbb{R}^+}$ is chiral.
\end{thm}
\begin{rem}
It was shown in \cite{kas-pow} that the unitary automorphisms of the weak$^\ast$-closed triple semigroup algebra $\mathcal{A}_{ph}$ are of the form $\operatorname{Ad}(V_t)$. It is still unknown if these are also the isometric isomorphisms of the algebra. In particular, it remains unclear to the author if the dual automorphisms of the norm closed algebra $\mathcal{A}_{ph}^{\mathbb{R}^+}$ can be extended to its weak$^\ast$-closure.
\end{rem}

\end{document}